\documentclass[11pt]{article}
\usepackage{amsmath,amsthm,amsfonts,amssymb,mathrsfs,bm}
\usepackage{multicol}
\usepackage{hyperref}
\usepackage{cite}
\hypersetup{colorlinks=true, urlcolor= black, linkcolor=black, citecolor=blue}
\usepackage[color=green, textsize=tiny]{todonotes}

\allowdisplaybreaks

\parskip \medskipamount
\newtheorem{theorem}{Theorem}[section]
\newtheorem{definition}[theorem]{Definition}

\newtheorem{lemma}[theorem]{Lemma}
\newtheorem{proposition}[theorem]{Proposition}

\newtheorem{remark}[theorem]{Remark}

\newtheorem{claim}[theorem]{Claim}

\newtheorem{conjecture}[theorem]{Conjecture}

\numberwithin{equation}{section}

\newcommand{\1}{{\text{\Large $\mathfrak 1$}}}

\newcommand{\til}{\widetilde}
\renewcommand{\phi}{\varphi}
\renewcommand{\epsilon}{\varepsilon}

\newcommand{\pr}[1]{\mathbb{P}\!\left(#1\right)}
\newcommand{\E}[1]{\mathbb{E}\!\left[#1\right]}

\newcommand{\prstart}[2]{\mathbb{P}_{#2}\!\left(#1\right)}
\newcommand{\prcond}[3]{\mathbb{P}_{#3}\!\left(#1\;\middle\vert\;#2\right)}
\newcommand{\econd}[2]{\mathbb{E}\!\left[#1\;\middle\vert\;#2\right]}

\def\Z{\mathbb{Z}}

\def\cU{\mathcal{U}}
\def\cT{\mathcal{T}}
\def\cP{\mathcal{P}}
\def\cG{\mathcal{G}}
\def\cF{\mathcal{F}}
\def\cE{\mathcal{E}}
\def\cC{\mathcal{C}}

\def\loc{\textrm{loc}}
\def\reg{\textrm{reg}}
\def\irr{\textrm{irr}}
\def\bad{\textrm{bad}}
\def\line{\textrm{line-good}}

\newcommand{\ce}[1]{\mathcal{C}_e(0;{#1})}
\newcommand{\ball}[2]{B(#1,#2)}
\newcommand{\Reg}[1]{\text{Reg}(#1)}

\newcommand{\cez}[1]{\mathcal{C}_e(z;{#1})}

\newcommand{\cpc}[2]{\mathrm{Cap}_{#1}(#2)}
\newcommand{\pcap}[1]{\mathrm{pCap}(#1)}

\newcommand{\diam}[1]{{\rm{diam}}(#1)}

\begin{document}
	
	\title{\bf Capacity in high dimensional critical percolation}
	
	\author{Amine Asselah \thanks{
			Universit\'e Paris-Est, LAMA, UMR 8050, UPEC, UPEMLV, CNRS, F-94010 Cr\'eteil; amine.asselah@u-pec.fr} \and
		Bruno Schapira\thanks{Universit\'e Claude Bernard Lyon 1, Centrale Lyon, INSA Lyon, Universit\'e Jean Monnet, CNRS, ICJ UMR 5208, 69622, Villeurbanne, France;  schapira@math.univ-lyon1.fr} \and Perla Sousi\thanks{University of Cambridge, Cambridge, UK;   p.sousi@statslab.cam.ac.uk} 
	}
	\date{}
	\maketitle
	
	\begin{abstract}  
		We introduce a notion of capacity for high dimensional critical percolation by showing that for any finite set $A$, the suitably rescaled probability that the cluster of $z$ intersects $A$ converges as $\|z\|\to\infty$. This can be viewed as a generalisation of the asymptotic of the two point function and we call the limit the p-capacity of $A$. We next show that the probability that the Incipient Infinite Cluster of $z$ intersects the set $A$ appropriately normalised is also of order the p-capacity of $A$ as $\|z\|\to\infty$. We conjecture that the p-capacity is of the same order as the $(d-4)$-Bessel-Riesz capacity and in support of this we estimate the p-capacity of balls. As a byproduct of our techniques we give a simpler proof of the one-arm exponent of Kozma and Nachmias for dimensions 8 and higher and as long as the two point function asymptotic holds. Our proofs make use of a new large deviations bound on the number of pioneers, which is the set of points on the boundary of a box which are part of the cluster of the origin restricted to this box.  				
		\bigskip
		

		\noindent MSC 2010 \emph{subject classifications.} Primary 60K35, Secondary 31B15.
	\end{abstract}

	\section{Introduction}
	Our aim in this paper is to introduce a notion of capacity in critical percolation on the Euclidean lattice in high dimension (the so-called mean-field regime), and investigate some of its properties. More precisely we shall consider critical Bernoulli percolation on $\mathbb Z^d$, either in dimension $d>10$, for the usual nearest-neighbor graph, or in any dimension $d>6$, but on a sufficiently spread-out graph, meaning that for a sufficiently large constant, all vertices of $\mathbb Z^d$ at distance smaller than this constant are initially connected by an edge. Recall that in percolation with parameter $p\in [0,1]$, edges are removed from the graph with probability $1-p$,  independently for each edge, and are declared open otherwise. It is known that in any dimension $d\ge 2$, there exists a critical parameter $p_c\in (0,1)$, such that for $p>p_c$, the origin is part of an infinite connected component of open edges with positive probability, and for $p<p_c$, all connected components are almost surely finite, see~\cite{Gri} for background on percolation. In this paper we only consider critical percolation, meaning that we consider the graph formed by open edges at $p=p_c$, and denote by $\mathbb P$ its law. Given $z\in \mathbb Z^d$, we let $\mathcal C(z)$ be the connected component of $z$, and given $A\subset \mathbb Z^d$, we write $z\longleftrightarrow A$, for the event that $z$ is connected to $A$, which we simply write $z\longleftrightarrow x$ when $A$ is reduced to a single point $\{x\}$. Likewise, $A\longleftrightarrow B$ denotes the event that two sets $A$ and $B$ are connected. Let also $\tau$ be the two point function: 
	$$\tau(z) = \mathbb P(0\longleftrightarrow z) = \mathbb P(z \in \mathcal C(0)). $$ 
		It is known by~\cite{FvdH17} for the nearest neighbor model in any dimension $d>10$ (see also~\cite{Hara08} for an earlier result valid in dimension $d\ge 19$) and by~\cite{HvdHS} for the sufficiently spread out model in any dimension $d>6$, that the two point function satisfies the following asymptotic: 
	\begin{equation}\label{twopoint}
		\tau(z) \sim a_d \cdot \|z\|^{2-d},
	\end{equation}
	for some constant $a_d>0$, as $\|z\| \to \infty$, where $\|\cdot \|$ denotes the Euclidean norm. For the rest of the paper we consider critical percolation satisfying~\eqref{twopoint} and $d> 6$. 
	
A question from Jason Miller served as the impetus for our work. Motivated by the analogy between branching random walks and critical percolation clusters in high dimensions, he asked whether one could define a notion of capacity for critical percolation similarly to the notion of branching capacity introduced by Zhu~\cite{Zhu}  in the context of branching random walks. 

The starting point of our study answers this question by introducing a notion of capacity in the setting of percolation and can be viewed as a generalisation of the two point function asymptotic~\eqref{twopoint}. 

	\begin{theorem}\label{thm.defpcap}
	Let $A$ be a finite subset of $\mathbb Z^d$. The following limit exists, and is called $p$-capacity of $A$:
	$$\textrm{pCap}(A) = \lim_{\|z\|\to \infty} \frac{\mathbb P(z\longleftrightarrow A)}{\tau(z)}.$$ 
\end{theorem}

Note the analogy with the definition of the Newtonian capacity, in the setting of random walks. Indeed, if $(X_n)_{n\ge 0}$ denotes a simple random walk on $\Z^d$, with $d\geq 3$, then the Newtonian capacity of a finite set $A\subset \mathbb Z^d$, is defined to be 
\[
\cpc{}{A} = \lim_{\|z\|\to\infty} \frac{\prstart{X[0,\infty)\cap A\neq \emptyset}{z}}{g(z)},
\]
where $g(z)= \sum_{n\ge 0} \mathbb P(X_n =z)$ is the Green's function and $X[0,\infty) = \{X_0,X_1,\dots\}$ is the range of the walk.  Zhu~\cite{Zhu} proved that the limit above exists when $X$ is replaced by a critical branching random walk in $\Z^d$, with $d\ge 5$ (at least when the underlying offspring distribution has finite second moment) and he called the limit the branching capacity of the set $A$. 

In Theorem~\ref{thm.twosets} below we extend the theorem above by replacing $z$ by $z+B$ for $B$ another finite set. Furthermore, in Remarks~\ref{remark.eqmeasure} and~\ref{rem.eqmes.ordering} we give some representations of the capacity in terms of equilibrium measures. 
	
By the definition of the p-capacity and~\eqref{twopoint} it immediately follows that it is translation invariant, i.e.\ for any finite $A$ and any $x\in \mathbb Z^d$, one has 
	$$\textrm{pCap}(x+A) = \textrm{pCap}(A).$$ 
	Note that it also follows from the definition that the $p$-Capacity of a point is one, and furthermore that the $p$-Capacity is increasing under set inclusion. By the inclusion-exclusion formula, it is immediate to see that it also satisfies 
	$$\textrm{pCap}(A\cup B) \le \textrm{pCap}(A) + \textrm{pCap}( B)  - \textrm{pCap}(A\cap B). $$ 
The main object of our study is to investigate how this notion of capacity naturally arises in various problems in critical percolation, at least in high dimension (the case of low dimension may also be particularly interesting, but it will not be addressed in this paper). 

Before stating our second main result, we recall the definition of the Incipient Infinite Cluster (abbreviated as IIC) rooted at a vertex $x$  and denoted by $\mathcal C_\infty(x)$. This is defined as the limit in law as $\|w\|\to\infty$ of the cluster of $x$, $\cC(x)$,  conditioned on the event that $\{x\longleftrightarrow w\}$. We refer the reader to~\cite{FvdH17,vdHJ,CCHS25} for a proof of this fact under our hypotheses. 

In our next result we show that the p-capacity of a set also controls the probability that the IIC cluster of a point far from the origin hits it. 
	
	\begin{theorem}\label{thm.IIC}
		There exists a positive constant $C>0$, such that for any finite $A\subset \mathbb Z^d$, 
		$$\frac 1C \cdot \textrm{pCap}(A) \le \liminf_{\|z\|\to \infty} \|z\|^{d-4}\cdot \mathbb P(\mathcal C_\infty(z)\cap A\neq \emptyset) \le 
		\limsup_{\|z\|\to \infty} \|z\|^{d-4} \cdot \mathbb P(\mathcal C_\infty(z)\cap A\neq \emptyset) \le C \cdot \textrm{pCap}(A) . $$ 
	\end{theorem}

	\begin{remark}\rm{Letting 
			$$\tau_\infty(z) = \mathbb P(z\in \mathcal C_\infty(0) ) ,$$ 
the theorem above shows that there exist two positive constants $c$ and $C$, such that for any $z\in \mathbb Z^d\setminus \{0\}$, 
			$$c\cdot \|z\|^{4-d} \le \tau_\infty(z) \le C \cdot \|z\|^{4-d}. $$
			This result had been proved in~\cite[Theorem~1.3]{vdHJ}  using lace expansion. Our proof of Theorem~\ref{thm.IIC} provides a new way of deriving this estimate, which does not rely on this technique (apart from using~\eqref{twopoint}, which at the moment is only proved using lace expansion, see however~\cite{DCP} for recent progress on the derivation of this estimate using a different argument). }
	\end{remark} 
	
	Motivated by the theorem above it is natural to conjecture the following stronger result.
	 
	\begin{conjecture}\label{conj.1} 
		For any finite $A\subset \mathbb Z^d$, 
		$$\lim_{\|z\|\to \infty} \frac{\mathbb P(\mathcal C_\infty(z)\cap A\neq \emptyset)}{\tau_\infty(z)} =  \textrm{pCap}(A). $$ 
	\end{conjecture}
	
	In the case of critical branching random walks, an analogous result holds. More precisely, Zhu~\cite{Zhu} proved that the branching capacity appears both as 
	the appropriately renormalised probability of hitting a set by a walk indexed by a critical tree and a tree conditioned to be infinite (which is the analogue of the IIC in this setup).  
	
	Using Theorem~\ref{thm.defpcap}, and the definition of the IIC measure, it is immediate to see that for any $z\in \mathbb Z^d$, 
	$$\mathbb P(\mathcal C_\infty(z) \cap A \neq \emptyset) = \lim_{\|w\|\to \infty} \frac{\mathbb P(z\longleftrightarrow A, w\longleftrightarrow z)}{\tau(w)} = \textrm{pCap}(A) + \textrm{pCap}(\{z\}) - \textrm{pCap}(A\cup \{z\}).$$
	Hence Conjecture~\ref{conj.1} is equivalent to the fact that 
	$$\lim_{\|z\|\to \infty}\frac{ \textrm{pCap}(A) +  \textrm{pCap}(\{z\}) -  \textrm{pCap}(A \cup\{z\})}{\tau_\infty(z)} = \textrm{pCap}(A).$$ 
	Unfortunately, we were not able to prove such a limiting result. However, in the setting of branching capacity and also general Bessel-Riesz capacities, in~\cite{ASS25} we were able to establish it and also extend it to a general finite set $B$ in place of $\{z\}$.  Naturally, this also leads us to the following stronger form of Conjecture~\ref{conj.1}:
	\begin{conjecture}\label{conj.2} For any finite $A,B\subset \mathbb Z^d$, 
		$$\lim_{\|z\|\to \infty}\frac{ \textrm{pCap}(A) +  \textrm{pCap}(B) -  \textrm{pCap}(A \cup(z + B))}{\tau_\infty(z)} = \textrm{pCap}(A)\cdot \textrm{pCap}(B).$$ 
	\end{conjecture}
	Again using the analogy between critical branching random walks and high dimensional percolation, it is natural to expect that the p-capacity and the branching capacity of a set should be of the same order, i.e.\ that their ratio should be bounded from above and below by universal constants only depending on dimension. In~\cite{ASS23}, it was shown that the branching capacity of a set is comparable to its $(d-4)$-capacity, denoted as $\textrm{Cap}_{d-4}$, which is an example of Bessel-Riesz capacity, and is defined via the variational formula for any finite set $A\subset \mathbb Z^d$ 
	\begin{equation}\label{var.C4}
		\textrm{Cap}_{d-4}(A) =  \Big (\inf \big\{\sum_{x,y\in A} (1+\|y-x\|)^{4-d} \mu(x)\mu(y) : \mu \textrm{ probability measure on }A\big\}\Big)^{-1}.
	\end{equation}
	Hence a fundamental problem in our setting would be to show that the same comparison holds between the $(d-4)$-capacity and the p-capacity of a set. It turns out that at least one direction is not difficult to prove. 
	\begin{proposition}\label{prop:lowercap}
		There exists a constant $c>0$, such that for any finite $A\subset \mathbb Z^d$, 
		$$\textrm{pCap}(A) \ge c \cdot \textrm{Cap}_{d-4}(A). $$ 
	\end{proposition}

	\begin{remark}\rm{
			The proof of the above proposition employs the second moment argument that originated in~\cite{FitzSal} and~\cite{Salis} and also appeared in~\cite{Martincap}, adapted to the percolation setup. 
			Actually, the proof provides a more quantitative result. It shows (see Lemma~\ref{lem:hitdiameterdistance}) that for any $\varepsilon>0$, there exists a constant $c>0$, such that for any finite set~$A$ containing the origin and any $z$ with $\|z\|\geq \varepsilon \cdot \max_{a\in A} \|a\|$, one has $$\mathbb P(z\longleftrightarrow A) \ge c \cdot \textrm{Cap}_{d-4}(A)\cdot \tau(z).$$ 
			}
	\end{remark} 
	Showing the converse inequality, namely that $\textrm{pCap}(A)$ is at most of order $\textrm{Cap}_{d-4}(A)$ for any finite~$A$, remains an important challenge. Nevertheless, building on ideas as in~Kozma and Nachmias~\cite{KN11}, we are able to prove the equivalence for balls, at least in dimension $8$ and more (the case of dimension $7$ will be handled in a forthcoming paper, using completely different methods~\cite{ASS26+}). Before stating the result let us fix some notation. For $z=(z_1,\dots,z_d)$, we let $\|z\|_\infty = \max_{i=1,\dots,d} |z_i|$ and for $r>0$ we set  
	\[
	B(0,r) = \{z\in \mathbb Z^d : \|z\|_\infty \le r\}.
	\]
	By the definition~\eqref{var.C4} it follows that $\textrm{Cap}_{d-4}(B(0,r))$ is of order $r^{d-4}$.  
	\begin{theorem}\label{thm.pcap.ball}
	Assume $d\ge 8$. 
		There exists a constant $C>0$, such that for any $r\ge 1$, 
		$$\textrm{pCap}(B(0,r)) \le C \cdot r^{d-4}. $$ 
	\end{theorem}
	Together with the general lower bound from Proposition~\ref{prop:lowercap} this completely characterises the order of magnitude of $\textrm{pCap}(B(0,r))$. On the other hand, it remains an interesting open problem to estimate the p-capacity of $m$-dimensional balls, with $m<d$, similarly to what is known for branching capacity~\cite{Zhu}.

The method of proof of Theorem~\ref{thm.pcap.ball} can also be used to give a simplified proof of the one arm exponent established in~\cite{KN11} as long as~\eqref{twopoint} holds and $d\geq 8$. We present this argument in Section~\ref{sec:newproof}, see also~\cite{EGPS26,PS26} for some other recent proofs (that appeared after the first version of this paper).

	In order to prove Theorems~\ref{thm.IIC} and~\ref{thm.pcap.ball}, we rely on the technique of regularity, which has been first introduced in the seminal paper of Kozma and Nachmias~\cite{KN11} for computing the one arm exponent, and has later found many other applications, see in particular~\cite{CH20, CHS23,CCHS25}. 
	However, the introduction of the p-capacity allows for some new conceptual shortcuts in their approach, see Section~\ref{sec:prelim} for some more elaborated discussion on this. Another important ingredient is a large deviations result for the number of pioneers on the boundary of a box, which we present now, and which allows for a new definition of a regular point, in purely geometric terms (see Definition~\ref{def.Kreg}), as opposed to the original one which involved some annoying conditional probability.   This type of large deviations estimates in percolation finds its roots in the tree graph inequality of Aizenman and Newman~\cite{AN84} (see also~\cite{A96}), and has been used to provide similar bounds for the size of a percolation cluster \textit{inside} a  ball. The original part of our result is to also control the number of pioneers in arbitrarily small balls on the boundary of the cube (in fact our proof works as well for the number of pioneers on a ball of a hyperplane).

To be more precise now, given $r,s>0$, and $x\in \partial B(0,r)$ (where by $\partial B(0,r)$ we mean the set of points in the ball $B(0,r)$ which have at least one neighbour outside of the ball), let $Q_s(x) = B(x,s)\cap \partial B(0,r)$. 
	We also define $\mathcal C_r(0)$ to be the set of points of the cube $B(0,r)$ that are connected to the origin using open paths lying entirely in the cube. Our result is as follows. 
	\begin{theorem}\label{thm.LD}
		There exist positive constants $c$ and $C$, such that for any $r,s,t\ge 1$ and any $x\in \partial B(0,r)$, 
		\begin{equation}\label{LD.strong}
		\mathbb P\Big(|\mathcal C_r(0)\cap Q_s(x)|\ge t\Big) \le \frac{C}{r^2} \cdot \exp (- \frac{ct}{s^2}). 
		\end{equation}
	\end{theorem}
	As already mentioned, our initial motivation for proving this result stems from the fact that it allows for a more convenient notion of regular points than in~\cite{KN11}. The prefactor $1/r^2$ in Theorem~\ref{thm.LD}, corresponds to the probability that the origin is connected to the boundary of the box $B(0,r)$ as shown in~\cite{KN11} (see Theorem~\ref{thm:onearm} below). However, in dimension $d\ge 8$, we can improve it, and get 
		instead the conjecturally optimal factor $s^{d-3}/r^{d-1}$, see our Theorem~\ref{thm.improvedLD}. 
		In the same vein we conjecture that the factor $s^2$ appearing in the exponential is optimal, in the sense that conditionally on hitting $Q_s(x)$, the size of the cluster $\mathcal C_r(0)$ in $Q_s(x)$ should be typically of order $s^2$.

	Now, another important comment is that using only the estimate $\tau(x,y)\asymp \|x-y\|^{2-d}$, one can get in a simple way the bound 
	\begin{equation}\label{LD.weak}
	\mathbb P\Big(|\mathcal C_r(0)\cap Q_s(x)|\ge t\Big) \le C\cdot \exp (- \frac{ct}{s^3}),
	\end{equation}
	without relying at all on the sophisticated results  of~\cite{KN11,CH20,CHS23}, see Remark~\ref{rem.simpleLD}. For the applications we have in mind, this weaker bound is sufficient in dimension $d\ge 8$. In particular we only use~\eqref{LD.weak} for proving Theorems~\ref{thm.pcap.ball} and~\ref{thm.improvedLD} as well as in our new proof of the one arm exponent. The stronger version~\eqref{LD.strong} is only used here for the proof of Theorem~\ref{thm.IIC}, but it may have wider applications. In particular it is used in~\cite{Romain} for proving the sharp lower bound on the two-point function restricted to a half-space.

	The proof of Theorem~\ref{thm.LD} heavily relies on bounds on the two point function in a half space which were obtained in~\cite{CH20,CHS23} and have been recently strengthened in~\cite{Romain}, see also Remark~\ref{rem.circular}. We further use a general argument from~\cite{AS24}, which transfers estimates on the two-point function into exponential large deviations bounds.

	We present now two additional applications of our results and techniques. The first one concerns the  probability of connection of two distant sets, and is a natural extension of Theorem~\ref{thm.defpcap}. 
	\begin{theorem}\label{thm.twosets}
		Let $A$ and $B$ be finite subsets of $\mathbb Z^d$. Then 
		$$\lim_{\|z\|\to \infty} \frac{\mathbb P\big(A \longleftrightarrow z+B\big)}{\tau(z)} = \textrm{pCap}(A) \cdot \textrm{pCap}(B). $$ 
	\end{theorem}
	
	Moreover, in Section~\ref{sec:connectingdistant} we provide a lower bound on the probability that two sets $A$ and $B+z$ are connected in terms of their $(d-4)$-capacities as long as $\|z\|$ is larger than a constant multiple of the diameter of $A\cup B$.

	In particular, by combining this result with Theorem~\ref{thm.pcap.ball}, one deduces a precise estimate for the probability that two distant balls are connected: it is of order their distance to the power $2-d$, times the product of their radii to the power $d-4$. Theorem~\ref{thm.twosets} is also a key ingredient of the proof in~\cite{Romain}. 
	Our second application provides another improvement of the one-arm estimate of Kozma and Nachmias~\cite{KN11}.  
	\begin{theorem}\label{thm.onearm}
		There exists $C>0$, such that for any finite  $A\subset\mathbb Z^d$, for any $r$ large enough, 
		$$\frac 1{Cr^2}\cdot \textrm{pCap}(A)\le  \mathbb P\big(A\longleftrightarrow \partial B(0,r)\big)\le \frac{C}{r^2}\cdot \textrm{pCap}(A). $$ 
	\end{theorem}
	We note that results of similar flavour have been obtained in the model of loop percolation in~\cite{vdHS, quirin}.  
	
	Finally, let us discuss another natural question, which is about the existence of an \textbf{equilibrium measure}. Ideally, given a finite set $A\subset \mathbb Z^d$, we 
	would like to define a probability measure $\mu$ on $A$, which should serve as a harmonic measure from infinity for the percolation cluster, i.e. given $a\in A$, we would like to define $\mu(a)$ as the limiting probability for a percolation cluster $\mathcal C(z)$ conditioned on hitting $A$, to hit it \textit{first} at $a$, as we let $z$ go to infinity. Of course the problem here is to make sense of the notion of first hitting point, for which there is no canonical choice. Note that another possible way to define the equilibrium measure, in the case of the p-capacity would also be defined via a variational formula similar to~\eqref{var.C4}, would be to take a minimizer. However, we were not able to prove any variational formula for $\textrm{pCap}$, and this remains an important open problem. Nevertheless, here we propose some possible choices of equilibrium measure, either using a random walk on the cluster (Proposition~\ref{prop:eqmeasure} and Remark~\ref{remark.eqmeasure} below), or using an arbitrary ordering of the elements of the set (Remark~\ref{rem.eqmes.ordering} below), in order to define its first hitting point. Of course, these are just two possible examples among many others, but our purpose here is just to illustrate in particular cases, that such probability measures supported on $A$ can effectively be defined. To be more precise now, given $z\in \mathbb Z^d$, and $A\subset \mathbb Z^d$, we denote by $H_A(z)$ the position of the simple random walk on $\mathcal C(z)$, starting from $z$, at its hitting time of $A$, on the event that $\mathcal C(z) \cap A \neq \emptyset$. 
	\begin{proposition}\label{prop:eqmeasure}
		Let $A$ be a finite subset of $\mathbb Z^d$, and $a\in A$. 
		The following limit exists: 
		$$e_A(a) = \lim_{\|z\| \to \infty} \frac{\mathbb P(H_A(z) = a,\, \mathcal C(z)\cap A\neq \emptyset)}{\tau(z)}.$$
		Furthermore, 
		$$\textrm{pCap}(A) = \sum_{a\in A} e_A(a). $$ 
	\end{proposition}

	\begin{remark}\label{remark.eqmeasure}
		\rm{Using reversibility of the random walk, our equilibrium measure may also be expressed in terms of the IIC. More precisely, given $a\in A$, consider $e_\infty(a,A)$ the first pivotal edge of $\mathcal C_\infty(a)$, for the event that $a$ is connected to infinity. In other words $e_\infty(a,A)$ is an edge of $\mathcal C_\infty(a)$, such that closing it disconnects $a$ from infinity, and no other edge of the (finite) connected component of $a$ in $\mathcal C_\infty(a) \setminus e_\infty(a,A)$ has this property.  
			Then the reversibility of the random walk shows that $e_A(a)$ is also equal to the probability that a random walk starting from $a$ restricted to $\mathcal C_\infty(a)$ hits first one endpoint of $e_\infty(a,A)$ before returning to $A$, see also Remark~\ref{rem.reversibility}. }
	\end{remark}
	
	\begin{remark} \label{rem.eqmes.ordering}
		\rm{Another natural example of equilibrium measure is the following. Consider an ordering of the elements of $A$, say $A=\{a_1,\dots,a_k\}$. 
			Then 
			$$e_A(a_i) = \mathbb P(\mathcal C_\infty(a_i) \cap \{a_1,\dots,a_{i-1}\} =\emptyset),$$ is such that $\textrm{pCap}(A) = \sum_{i=1}^k e_A(a_i)$, and is obtained as the limit 
			$$e_A(a_i) = \lim_{\|z\|\to \infty} \frac{\mathbb P(a_i \in \mathcal C(z)\textrm{ and } a_j\notin \mathcal C(z), \textrm{ for all }j<i)}{\tau(z)}. $$  
			Actually, this representation may also be used to provide a proof of the existence of the limit in Theorem~\ref{thm.defpcap}, which would be slightly shorter than the one we present in Section~\ref{sec.proofdefcap}, but we feel that the proof we give there in terms of last pivotal edges is also interesting on its own.}
	\end{remark}

\textbf{Organisation}:  In Section~\ref{sec:prelim} we recall known results in percolation and give an overview of the proof ideas. 
In Section~\ref{sec.proofdefcap} we present the proof of Theorem~\ref{thm.defpcap} on the existence of the p-capacity of a set and we also give the proof of Proposition~\ref{prop:eqmeasure}. Then in Section~\ref{sec.LD} we present a general framework and state a general large deviations result which implies Theorem~\ref{thm.LD} as a particular case. In Section~\ref{sec.regpoints} we define our new notion of regular point and show using our large deviations result and essentially the same exploration argument as in~\cite{KN11} that most pioneers are regular points with very high probability. In Section~\ref{sec:hitting} we give an alternative proof of the main ingredient in the proof of the one arm exponent from~\cite{KN11} and we also present the proof of Theorem~\ref{thm.IIC}. In Section~\ref{sec:pCapball} we prove Theorem~\ref{thm.pcap.ball} and finally in Section~\ref{sec:connectingdistant} we give the proofs of Theorems~\ref{thm.twosets} and~\ref{thm.onearm}.

\textbf{Notation}: Given two functions $f$ and $g$, we write $f\lesssim g$, if there exists a constant $C>0$, such that $f(x) \le Cg(x)$, for all $x$, and $f\asymp g$, if $f\lesssim g$ and $g\lesssim f$, and $f\sim g$, if the ratio $f(x)/g(x)$ tends to one at infinity.

We write $A\stackrel{\Lambda}{\longleftrightarrow} B$ for the event that two sets $A,B\subset \mathbb Z^d$ are connected by an open path remaining in a subset $\Lambda\subset \mathbb Z^d$.  
	
\section{Preliminaries and overview}
 \label{sec:prelim} 	
An event is increasing if it is stable under the operation of  opening new edges. 
We recall the Harris/FKG inequality (see Theorem (2.4) in~\cite{Gri}) which asserts that for any two increasing events $E$ and $F$ we have 
\begin{equation}\label{FKGineq}
		\pr{E\cap F}\geq \pr{E}\cdot \pr{F}.
\end{equation}
	A subgraph $A$ is said to witness an increasing event $E$, if 
	whenever all edges of the graph $A$ are open, the event $E$ holds. 
	Given two increasing events $E$ and $F$, we say that they occur disjointly and write $E\circ F$ if there exist two disjoint subgraphs $A$ and $B$ that witness $E$ and $F$ respectively and whose edges are all open. The BK inequality asserts that for any two increasing events $E$ and~$F$, one has (see Theorem (2.12) in~\cite{Gri}) 
	\begin{equation}\label{BKineq}
		\mathbb P(E\circ F) \le \mathbb P(E)\cdot \mathbb P(F).
	\end{equation}
	
	We also recall the following result due to Kozma and Nachmias~\cite{KN11} on the one-arm estimate in critical percolation, see also~\cite{EGPS26, PS26} for two recent alternative proofs, and~\cite{BRH26, CCFK26} for recent improvements obtaining a true asymptotic result.
	\begin{theorem}\label{thm:onearm}\rm{(\cite[Theorem~1]{KN11})}
		Let $d>6$. For the nearest neighbour model for which the two point function satisfies~\eqref{twopoint} and for the sufficiently spread out percolation model we have 
		\begin{equation}\label{onearm}
			\mathbb P(0\longleftrightarrow \partial B(0,r)) \asymp \frac 1{r^2}. 
		\end{equation}
	\end{theorem}

\textbf{Overview:} One of the estimates that appears in many proofs is a lower bound on the probability that a sufficiently far away point $z$ is connected to $\cC_r(0)$, where $\cC_r(0)$ is the cluster of $0$ created by open bonds lying  inside the box~$\ball{0}{r}$. In order for this to hold, $z$ must be connected to $\cC_r(0)\cap \partial \ball{0}{r}$, the set of the so-called \textit{pioneers}. We can further restrict to ``regular pioneers", $\Reg{\cC_r(0)}$. These are the pioneers for which (i) the density of vertices of $\cC_r(0)$ in large balls around them is typical, and (ii) the density of other pioneers is also typical. This latter condition requires a control on the surface large deviations of pioneers. This notion of regularity first appeared in the breakthrough work of Kozma and Nachmias~\cite{KN11} and it also appeared in subsequent works, such as~\cite{CH20, CHS23, CCHS25, sapoz}. 	It involved an intricate conditioning which we have now bypassed. Moreover, the requirement of a surface low density of pioneers is new. 
	
	Let us now outline the estimate on $\pr{z\longleftrightarrow \Reg{\cC_r(0)}}$. Typically, when $\cC_r(0)$ intersects $\partial \ball{0}{r}$, then $|\Reg{\cC_r(0)}|\asymp r^2$. In such a case, we call the cluster $\cC_r(0)$ good. We want to show that 
	\begin{align*}
		\prcond{z\longleftrightarrow  \Reg{\cC_r(0)}}{\cC_r(0) \text{ is good}}{} \gtrsim \pr{z\longleftrightarrow 0} \cdot \econd{|\Reg{\cC_r(0)}|}{\cC_r(0) \text{ is good}}{} \asymp \tau(z) \cdot r^{2}.
	\end{align*} 
Summing over all possible clusters we get 
	\begin{align}\label{eq:firststeph}
	\begin{split}
	\pr{z\longleftrightarrow \Reg{\cC_r(0)}, \cC_r(0) \text{ is good}}&= \sum_{H \text{ good}} \pr{\cC_r(0)=H} \prcond{z\longleftrightarrow \Reg{H}}{\cC_r(0)=H}{} \\
		&\geq \sum_{H\text{ good}} \pr{\cC_r(0)=H} \prcond{z\stackrel{\text{off } H}\longleftrightarrow \Reg{H}}{\cC_r(0)=H}{} 
	\\&=\sum_{H \text{ good}}\pr{\cC_r(0)=H} \pr{z\stackrel{\text{off } H}\longleftrightarrow \Reg{H}},
	\end{split}
	\end{align}
	where we write $\{x\stackrel{\text{off } A}\longleftrightarrow y\}$ for the event that $x$ and $y$ are connected via an open path of edges not using any vertices in $A$. Note that above we were able to decorrelate the events and get the last equality because the two events under consideration depend on disjoint sets of edges. This idea of decorrelation has already appeared extensively in previous works, in particular in~\cite{KN11}. 
	One then writes
	\begin{align}\label{eq:firstequationzreg}
	\pr{z\stackrel{\text{off } H}\longleftrightarrow \Reg{H}} = \pr{z\longleftrightarrow \Reg{H}} - \pr{z\stackrel{\text{via } H}\longleftrightarrow \Reg{H}},
	\end{align}
	where we write $\{x\stackrel{\text{via } A}\longleftrightarrow y\}$ for the event that every open path connecting $x$ and $y$ uses at least one vertex in~$A$. 
	To deal with the first term in the difference above we use that the p-capacity is lower bounded by the $(d-4)$-capacity. This lower bound is  both simple and quantitative. More precisely, we prove, using an idea of Fitzsimmons and Salisbury~\cite{FitzSal} in the setting of hitting times for Markov chains, that for $\|z\|\geq 2\diam{A}$, 
	\[
	\pr{z\longleftrightarrow A}\gtrsim \tau(z)\cdot  \cpc{d-4}{A}.
	\]
	In case the set $A$ is ``sparse'', then its $(d-4)$-capacity is shown to be of order its volume (see Claim~\ref{cl:lowerboundoncapacity}). This follows from the variational characterisation~\eqref{var.C4}. Note that for sparse sets the above lower bound matches, up to constant, the upper bound which follows from a union bound. Our regularity condition ensures that when $H$ is good, the set $\Reg{H}$ is sparse. Therefore, 
	\[
	\pr{z\longleftrightarrow \Reg{H}} \gtrsim \tau(z) \cdot \cpc{d-4}{\Reg{H}} \asymp \tau(z) \cdot r^2.
	\]
	It remains to show that 
	\[
	\pr{z\stackrel{\text{via } H}\longleftrightarrow \Reg{H}} \ll \tau(z) \cdot r^2.
	\]
	By first using a union bound and then the BK inequality one gets
	\begin{align}\label{eq:bigsumappear}
	\nonumber\pr{z\stackrel{\text{via } H}\longleftrightarrow \Reg{H}}\leq \sum_{x\in \Reg{H}}\pr{z\stackrel{\text{via } H}\longleftrightarrow x}\leq \sum_{x\in \Reg{H}}\sum_{w\in H} \tau(z,w) \tau(w,x) \\
	\asymp \tau(z) \sum_{\substack{x\in \Reg{H}\\ w\in H}} \tau(w,x),
	\end{align}
where for the last step we used that $\|z\|$ is large. To show the above sum is much smaller than $r^2$ one needs a control on the density of sites 
of the cluster in concentric annuli around regular points. The case of annuli with sufficiently large radii is handled using the definition of regularity, but the case of small radii requires a separate argument. Here we also bypass the original idea from~\cite{KN11} and replace it by a simpler approach. 
We introduce the notion of  \textit{line-good} points and instead of $\{z\longleftrightarrow \Reg{H}\}$ we consider the event that $z$ is connected to the set of line-good points.


\section{Existence of p-capacity} 
\label{sec.proofdefcap}
In this section we prove Theorem~\ref{thm.defpcap} and Proposition~\ref{prop:eqmeasure}. 
	
	Let us start with the proof of Theorem~\ref{thm.defpcap}. Let $A$ be a finite subset of $\mathbb Z^d$, and let $z\in \mathbb Z^d\setminus A$. For any edge $e$ of the graph under consideration (either the nearest neighbor graph if $d>10$ or a sufficiently spread out graph if $d>6$), we consider the event 
	$$\mathcal P_{z,A}(e) = \{z\longleftrightarrow A\} \cap \{e \textrm{ is the last pivotal edge for the event that $z$ is connected to $A$}\}. $$  
	In other words, $\mathcal P_{z,A}(e)$ is the event that $z$ is connected to $A$, the edge $e$ is open and part of the cluster of $z$, but if one closes it, then $z$ becomes disconnected from $A$, and furthermore, for any other edge $f$ with this property, closing $f$ also disconnects $e$ from $z$. 
	Now we let 
	$$\mathcal P_{z,A}= \bigcup_{e} \mathcal P_{z,A}(e),$$
	where the union is over all the edges of the graph (note that by definition, one can always restrict the union to edges having at least one endpoint which is not in $A$. Hence $\mathcal P_{z,A}$ is the event that there is at least one pivotal edge for the event that $z$ is connected to $A$. 
	The first step in the proof is the following lemma. 
	\begin{lemma}\label{lem:zapivotal}
		One has 
		$$\lim_{\|z\|\to \infty} \frac{\mathbb P\big(\{z\longleftrightarrow A\} \cap \mathcal P_{z,A}^c \big) }{\tau(z)} = 0. $$
	\end{lemma}
	\begin{proof}[\bf Proof]
		The proof follows from the observation that if $z$ is connected to $A$, but there is no pivotal edge, then there must exist two disjoint open paths from $z$ to $A$. Hence, using the BK inequality~\eqref{BKineq}, we get 
		$$\mathbb P\big(\{z\longleftrightarrow A\} \cap \mathcal P_{z,A}^c \big) \le \mathbb P(\{z\longleftrightarrow A\} \circ \{z\longleftrightarrow A\}) 
		\le \mathbb P(z\longleftrightarrow A)^2 \le \big(\sum_{x\in A} \tau(x,z)\big)^2 \le C |A|^2 \cdot \tau(z)^2,$$
		for some constant $C>0$. The result follows using that $\lim_{\|z\|\to \infty} \tau(z)= 0$, e.g. by~\eqref{twopoint}. 
	\end{proof}
	The next step is given by the following lemma. For an edge $e=\{x,y\}$ we let $\|e\| = \max(\|x\|,\|y\|)$.
	\begin{lemma}\label{lem:pCap2}
		Let $A$ a be finite subset of $\mathbb Z^d$. There exists a constant $C>0$, such that for any $K\ge 1$ and any $z\in \mathbb Z^d$, with $\| z\|\ge K$,  
		$$\frac 1{\tau(z)} \sum_{e:\|e\|\ge K} \mathbb P(\mathcal P_{z,A}(e)) \le \frac{C}{K^{d-4}}. $$ 
	\end{lemma}
	\begin{proof}[\bf Proof]
		Note that it suffices to prove the result for $K$ large enough, since for any edge $e$, one has $\mathbb P(\mathcal P_{z,A}(e))\le \mathbb P(z\longleftrightarrow A) \le c|A|\tau(z)$, for some constant $c>0$ (independent of $e$ and $z$). In particular one can assume that both endpoints of the edges under consideration are at distance at least $K/2$ from the origin. Now observe that if $e=\{x,y\}$ is such an edge, 
		then one has
		$$\mathcal P_{z,A}(e) \subseteq \big(\{x\longleftrightarrow A\}\circ  \{x\longleftrightarrow A\} \circ \{y\longleftrightarrow z\}\big)   \cup 
		\big(\{y\longleftrightarrow A\}\circ  \{y\longleftrightarrow A\} \circ \{x\longleftrightarrow z\} \big), $$ 
		and using~\eqref{twopoint} we obtain that 
		$$\mathbb P(\mathcal P_{z,A}(e)) \lesssim  |A|^2 \cdot \tau(x)^2\cdot \tau(z-x) \lesssim  |A|^2 \cdot \frac 1{\|x\|^{2(d-2) }\cdot (1+\|z-x\|^{d-2})}. $$ 
		Consequently, using that $d\ge 5$, 
		$$\frac 1{\tau(z)} \sum_{e:\|e\|\ge K} \mathbb P(\mathcal P_{z,A}(e)) \lesssim |A|^2   \sum_{x:\|x\|\ge K/2} \frac {\|z\|^{d-2}}{\|x\|^{2(d-2)}\cdot (1+\|z-x\|^{d-2})} \lesssim  \frac{|A|^2}{K^{d-4}}$$
		and this finishes the proof. 
	\end{proof}
	The previous lemma ensures that one can deal with only a finite number of edges. It remains to see that for these edges, the probabilities of the events $\mathcal P_{z,A}(e)$ converge as we let $z$ go to infinity. The idea for this is to first localize these events, and then use that the conditional probabilities of local events converge, by construction of the IIC measures. To be more precise now, given an edge $e$ and $x\in \mathbb Z^d$, we let $\mathcal C^e(x)$ be the connected component of $x$ after one closes $e$, and similarly let $\mathcal C_\infty^e(x)$ be  
	the connected component of $x$ in $\mathcal C_\infty(x)$, after one closes $e$. 
	Also, given an oriented edge $\vec{e}=(x,y)$, we define (with $e$ the unoriented edge $\{x,y\}$), 
	\begin{align*}
	\mathcal P_{\infty,A}(\vec e) &  = \{|\mathcal C_\infty^e(y)|=\infty, \mathcal C_\infty^e(y) \cap A= \emptyset\} \cap \{\mathcal C_\infty(y)\textrm{ contains two disjoint paths from $x$ to $A$}\}. 
	\end{align*} 
	
	Finally given an unoriented edge $e=\{x,y\}$, we let $\vec e_1 = (x,y)$ and $\vec e_2=(y,x)$ denote the two associated oriented edges. One can now state our final lemma. 
	\begin{lemma}\label{lem:pCap3}
		Let $A$ be a finite subset of $\mathbb Z^d$. For any edge $e$, one has 
		$$\lim_{\|z\|\to \infty} \frac {\mathbb P(\mathcal P_{z,A}(e))}{\tau(z)} = \mathbb P(\mathcal P_{\infty,A}(\vec e_1)) + \mathbb P(\mathcal P_{\infty,A}(\vec e_2)). $$ 
	\end{lemma}
	\begin{proof}[\bf Proof]
		For $B\subseteq \mathbb Z^d$ and $u\in \mathbb Z^d$ we denote by $\mathcal C(u;B)$ the connected component of $u$ when we just consider percolation on $B$ (i.e. the set of points that can be reached from $x$ by an open path that remains entirely in $B$), and by $\mathcal C^e(u;B)$ the connected component of $u$ in $\mathcal C(u;B)$ after one closes the edge $e$. Then given an oriented edge $\vec e=(x,y)$, and $R>\|e\|$, we let 
		\begin{align*}
			\mathcal P_{z,A}^R(\vec e) & = \{\mathcal C^e(x) \subseteq B(0,R) \textrm{ and } \mathcal C^e(x) \textrm{ contains two disjoint paths from }x \textrm{ to } A\} \\
			& \quad \cap \{\mathcal C^e(y;{B(0,R)})\cap A = \emptyset\}\cap \{x\longleftrightarrow z\}. 
		\end{align*}
		We also define 
		\begin{equation*}
			\mathcal P_{\infty,A}^R(\vec e) = \left\{ \begin{array}{cc}
				\mathcal C_\infty^e(x) \textrm{ contains two disjoint paths from }x \textrm{ to } A, \ \mathcal C_\infty^e(x) \subseteq B(0,R), \\
				\mathcal C_\infty(x) \textrm{ does not contain any path from $y$ to $A$ that avoids $e$ and remains in }B(0,R)
			\end{array}
			\right\}. 
		\end{equation*}
		For any $R>\|e\|$, we note that the event $\mathcal P_{z,A}^R(\vec e)$ is the intersection of one event that only depends on the state of the edges inside $B(0,R)$ and the event that $x$ is connected to $z$ by an open path. 
		Hence, the existence of the IIC measure implies that 
		\begin{equation*}
			\lim_{\|z\|\to \infty} \frac{\mathbb P( \mathcal P_{z,A}^R(\vec e))}{\tau(z)} = \mathbb P\big( \mathcal P_{\infty,A}^R(\vec e)  \big). 
		\end{equation*} 
		Given an unoriented edge $e$, we let 
		$$\mathcal P_{z,A}^R(e)= \mathcal P_{z,A}^R(\vec e_1) \cup \mathcal P_{z,A}^R(\vec e_2), $$
		where we recall that $\vec e_1$ and $\vec e_2$ are the two oriented edges associated to $e$.  
		Note that the two events on the right hand side above are disjoint when $\|z\|>R$. Thus we also have 
		\begin{align}\label{eq:directededge}
			\lim_{\|z\|\to \infty} \frac{\mathbb P( \mathcal P_{z,A}^R(e))}{\tau(z)} = \mathbb P\big( \mathcal P_{\infty,A}^R(\vec e_1)  \big)+\mathbb P\big( \mathcal P_{\infty,A}^R(\vec e_2)  \big).
		\end{align}
		Next using that $\cC_\infty(x)$ is almost surely one-ended~\cite[Theorem~1.3(ii)]{vdHJ}, we see that $ \mathcal P_{\infty,A}(\vec e)$ is the increasing union over $R$ of the events $ \mathcal P_{\infty,A}^R(\vec e)$. Therefore, one has for any oriented edge $\vec e$, 
		\begin{align}\label{eq:limitR}
			\lim_{R\uparrow \infty} \mathbb P(\mathcal P_{\infty,A}^R(\vec e)) = \mathbb P(\mathcal P_{\infty,A}(\vec e)). 
		\end{align}
		To finish the proof of the lemma it suffices to show that
		\begin{align}\label{eq:goallimit0}
			\lim_{R\to\infty}	\lim_{\|z\|\to\infty}	\frac{\pr{\cP_{z,A}(e)\cap \cP_{z,A}^R(e)^c}}{\tau(z)} = 0 \quad \text{ and } \quad \lim_{R\to\infty}\lim_{\|z\|\to\infty}  \frac{\pr{\cP_{z,A}(e)^c\cap \cP_{z,A}^R(e)}}{\tau(z)} = 0.
		\end{align}
		Indeed, once this is proved then the proof can be concluded by writing 
		\begin{align*}
			\lim_{\|z\|\to\infty}\frac{	\pr{\cP_{z,A}(e)}}{\tau(z)} =\lim_{R\to\infty}\lim_{\|z\|\to\infty} \frac{	\pr{\cP^R_{z,A}(e)}}{\tau(z)} 
		\end{align*}
		and using also~\eqref{eq:directededge} and~\eqref{eq:limitR}. We now prove~\eqref{eq:goallimit0}. Let $R>2\|e\|^2$. Then 
		we observe that 
		$$
		\mathcal P_{z,A}(e)\cap \mathcal P_{z,A}^R(e)^c  \subseteq \Big(\{x\longleftrightarrow \partial B(0,R) \} \circ \{y\longleftrightarrow z\}\Big)\cup \Big(\{y\longleftrightarrow \partial B(0,R)\} \circ \{x\longleftrightarrow z\}\Big).$$ 
		Therefore by the BK inequality~\eqref{BKineq} and the one arm estimate~\eqref{onearm}, we get that for some constant $C>0$, and for any $z$ with $\|z\|\ge 2\|e\|$, 
		$$\mathbb P\big(\mathcal P_{z,A}(e) \cap \mathcal P_{z,A}^R(e)^c \big) \le \frac{C}{R^2}\cdot \tau(z),$$
		and hence this proves that the first limit in~\eqref{eq:goallimit0} is equal to $0$. 
		For the second limit of~\eqref{eq:goallimit0} we first have the inclusion when $\|z\|>R$,
		\begin{align*}
			  \mathcal P_{z,A}^R(\vec e_1) \cap \mathcal P_{z,A}(e)^c  \subseteq & \{\textrm{there exists an open path from $y$ to $A$ that goes outside $B(0,R)$}\} \cap \{y\longleftrightarrow z\}\\
			 \subseteq &\Big(\bigcup_{ \|w\|\le \sqrt R} \{y\longleftrightarrow w\} \circ \{w\longleftrightarrow z\} \circ \{w\longleftrightarrow \partial B(0,R)\}\Big) \\
			&\cup \Big(\bigcup_{\|w\|>\sqrt R}   \{y\longleftrightarrow w\} \circ \{w\longleftrightarrow z\} \circ \{w\longleftrightarrow A\}\Big),
		\end{align*}
		which entails as above, using a union bound and the BK inequality~\eqref{BKineq}
		\begin{align*}
			&\mathbb P\big(\mathcal P_{z,A}^R(\vec e_1) \cap \mathcal P_{z,A}(e)^c\big) \\& \lesssim  \sum_{\|w\|\le \sqrt R} \frac{1}{(1+\|w-y\|^{d-2}) \cdot \|z\|^{d-2} \cdot R^2} +  \sum_{\|w\|> \sqrt R} \frac{|A|}{\|w\|^{2(d-2)} \cdot (1+\|w-z\|^{d-2})}\\
			& \lesssim \frac{|A|\cdot \tau(z)}{R},
		\end{align*}
		and the same bound holds if we replace $\vec e_1$ by $\vec e_2$. This finally establishes that the second limit in~\eqref{eq:goallimit0} is equal to $0$ and concludes the proof of the lemma.
	\end{proof}

	One can now conclude the proof of the existence of the p-capacity. 
	
	\begin{proof}[\bf Proof of Theorem~\ref{thm.defpcap}] 
		By Lemmas~\ref{lem:zapivotal} and~\ref{lem:pCap2} we get 
		\[
		\lim_{\|z\| \to \infty} \frac{\mathbb P(z\longleftrightarrow A)}{\tau(z)} = \lim_{\|z\| \to \infty} \frac{\mathbb P( \cP_{z,A})}{\tau(z)} =\lim_{K\to\infty}\sum_{e:\|e\|<K}\lim_{\|z\| \to \infty} \frac{\mathbb P(\cP_{z,A}(e))}{\tau(z)}.
		\]
		Applying Lemma~\ref{lem:pCap3} we then deduce 
		\begin{equation}\label{expr.pivotal.pcap}
			\lim_{\|z\| \to \infty} \frac{\mathbb P(z\longleftrightarrow A)}{\tau(z)} = \sum_{\vec e} \mathbb P(\mathcal P_{\infty,A}(\vec e)),
		\end{equation}
		where the sum on the right hand side runs over all oriented edges of the graph. To conclude it just remains to see that this last sum is finite. For this one can use Fatou's lemma together with Lemma~\ref{lem:pCap2} for $K=1$ and also Lemma~\ref{lem:pCap3}.  
	\end{proof}

	
	\begin{proof}[\bf Proof of Proposition~\ref{prop:eqmeasure}]
		The proof of this proposition proceeds as for the proof of Theorem~\ref{thm.defpcap}, with only minor changes. Fix $A$ a finite subset of $\mathbb Z^d$. 
		Given $a\in A$, $z\in \mathbb Z^d$, and some oriented edge $\vec e$ let us define the events (with $e$ denoting the associated unoriented edge), 
		$$\mathcal P_{z,A}(e,a) = \mathcal P_{z,A}(e) \cap \{H_A(z) = a\}, \quad\textrm{and}\quad  \mathcal P_{\infty,A}(\vec e,a)  = \mathcal P_{\infty,A}(\vec e) \cap \{ H_A(x)<\infty\}\cap \{H_A(x) = a\},$$
		with the notation from the previous section, and with, we recall, $H_A(z)$ the first hitting point of $A$ by a simple random walk on the cluster $\mathcal C(z)$ starting from $z$. Observe that on the event when $e=(x,y)$ is a pivotal edge for the event that $z$ is connected to $A$, with say $x$ still connected to $A$ when we close $e$ (but not $y$), the the law of $H_A(z)$ only depends on the configuration of the cluster $\mathcal C^e(x)$, i.e. the cluster of $x$ after we close $e$ (and is equal to the law of the first hitting point of $A$ by a random walk on $\mathcal C^e(x)$ starting from $x$). Hence the whole proof of Lemma~\ref{lem:pCap3} applies here as well and shows that for any fixed edge $e$ and any $a\in A$, 
		$$\lim_{\|z\|\to \infty} \frac {\mathbb P(\mathcal P_{z,A}(e,a))}{\tau(z)} = \mathbb P(\mathcal P_{\infty,A}(\vec e_1,a)) + \mathbb P(\mathcal P_{\infty,A}(\vec e_2,a)). $$ 
		We then conclude the proof exactly as in the previous section. \end{proof}

	\begin{remark}\label{rem.reversibility}\rm{
			Note that by reversibility of the simple random walk, for any oriented edge $\vec e=(x,y)$, on the event $\mathcal P_{\infty,A}(\vec e)$, the probability that a simple random walk on $\mathcal C_\infty(y)$ starting from $x$ hits first $A$ at some point $a\in A$, is the same as the probability that a simple random walk on this cluster, and starting from $a$, hits first $x$ before returning to $A$. Furthermore, it is not difficult to see that $\mathcal C_\infty(y)$ and $\mathcal C_\infty(a)$, conditionally on the event that $y$ and $a$ are connected, have the same law.  
			Hence $e_A(a)$ is also equal to the probability that a simple random walk on $\mathcal C_\infty(a)$ hits first the first edge that disconnects $a$ from infinity, before returning to $A$, as mentioned in Remark~\ref{remark.eqmeasure}. }
	\end{remark}
	

	\section{Large deviations for the number of pioneer points}\label{sec.LD}
In this section we prove Theorem~\ref{thm.LD}. 
We start by stating a general result, ensuring that some exponential moment of a functional of the occupation field of a cluster is finite, under some technical hypotheses on the two point function. This result can be seen as a generalization of Theorem 1.2 in~\cite{AS24}, and finds its roots both in the tree-graph inequality of Aizenman and Newman~\cite{AN84}, and in Kac's moment formula for local times of random walks, see~\cite[Proposition 2.9]{Sz12}.

		\begin{definition}\label{def.cluster}
		\rm{
			Let $\Lambda\subset \mathbb Z^d$ be a (non necessarily finite) subset of vertices. We denote by $\tau_\Lambda(x,y) = \mathbb P(x\stackrel{\Lambda}{\longleftrightarrow} y)$, the restricted two-point function in $\Lambda$ and for $x\in \Lambda$ we write $\cC(x;\Lambda)$ for the cluster of~$x$ created by the open bonds with both endpoints in the set $\Lambda$. 
		}
	\end{definition}

Given $h:\Lambda\times \Lambda \to [0,\infty)$ and $\varphi : \Lambda \to [0,\infty)$, we write $h*\varphi(x) = \sum_{y\in \Lambda} h(x,y)\varphi(y)$, for $x\in \Lambda$. 
	\begin{theorem}\label{thm.generalLD}
		Let $\Lambda\subset \mathbb Z^d$ and $\varphi: \Lambda \to [0,\infty)$ be given. Assume that there exists a function $h:\Lambda\times \Lambda \to [0,+\infty)$, and a constant $C>0$, such that the following three hypotheses hold:
		\begin{enumerate}
			\item $\tau_\Lambda(x,y) \le h(x,y), \quad \textrm{for all }x,y\in \Lambda$, 
			\item $\sup_{x\in \Lambda} h*\varphi (x) \le 1$, 
			\item $\sum_{y\in \Lambda} h(x,y)\cdot (h*\varphi(y))^2 \le C\cdot  h*\varphi(x), \quad \textrm{for all }x\in \Lambda$.  
		\end{enumerate}
		Then there exists a constant $\lambda>0$ (only depending on the constant $C$, but not on the particular choice of the function $\varphi$), such that for any $x\in \Lambda$, 
		$$\mathbb E\Big[\exp\Big(\lambda\cdot\sum_{y\in  \cC(x;\Lambda)} \varphi(y)\Big)\Big] \le 1+ h*\varphi(x). $$ 
	\end{theorem}
	\begin{proof}[\bf Proof]
		The proof is exactly the same as the one of Theorem 1.2 in~\cite{AS24}. Briefly, the main idea is to show by induction on $n\ge 1$, that for some constant $K>0$, one has for all $n\ge 1$ and all $x\in \Lambda$, (with $a\vee b = \max(a,b)$),  
		\begin{equation}\label{HR.expmoment}
			\mathbb E\Big[\Big(\sum_{y\in \cC(x;\Lambda)} \varphi(y)\Big)^n \Big] \le K^{n-1} \cdot (1\vee (n-2))! \cdot h*\varphi(x).
		\end{equation}
		Indeed once this is proved, the result follows by summation over $n$. 
		Note that the case $n=1$ in~\eqref{HR.expmoment} directly follows from the first hypothesis of the theorem. Then the induction step can be shown using the BK inequality and the two other hypotheses. We refer the reader to~\cite{AS24} for details. 
	\end{proof}
	
	Let us now explain how this result may be used to prove Theorem~\ref{thm.LD}. We will in fact prove bounds on the half space first, and then naturally deduce the desired estimate on a cube, just using that the percolation on a cube is stochastically dominated by percolation on a half-space. So let $\mathcal H = \{z\in \mathbb Z^d : z_1\ge 0\}$, and for $s>0$,  let
	$Q_s = B(0,s)\cap \partial \mathcal H$, where we set $\partial \mathcal H = \{z:z_1= 0\}$ (note the slight abuse of notation here in the case of a spread-out graph, where with our definition $\partial \mathcal H$ is not the boundary of $\mathcal H$ in the usual sense). 
	We now fix $s>0$, and define the function $\varphi : \mathcal H\to \mathbb R_+$, by 
	$$\varphi(y) = \frac {\1(y\in Q_s)}{s^2}.$$ 
	Next, for $x,y\in \mathcal H$, let $$r_{x,y} = d(x,\partial \mathcal H)\wedge \|x-y\|,$$
	where we write $d(y,A)$ for the distance from $y$ to a set $A$, and $a\wedge b = \min(a,b)$.  
	Let then $h:\mathcal H \times \mathcal H \to \mathbb R_+$ be defined by 
	$$h(x,y) = \frac{(1+r_{x,y})\cdot (1+r_{y,x})}{1+\|x-y\|^d}, \quad x,y\in \mathcal H,$$ 
	All the estimates we need are gathered in the following lemma. 
	\begin{lemma}\label{lem.LD}
		There exists a constant $C>0$ (independent of $s$), such that for any $x,y\in \mathcal H$, 
		\begin{enumerate}
			\item $\tau_{\mathcal H}(x,y) \le C\cdot h(x,y)$ 
			\item  $h*\varphi(y) \le C$. 
			\item $ \sum_{z\in \mathcal H} h(y,z)\cdot  (h*\varphi(z))^2 \le C\cdot h*\varphi(y)$.  
		\end{enumerate}
	\end{lemma}
	
	\begin{proof}[\bf Proof of Lemma~\ref{lem.LD}] 
		
		The first item has been proved by Panis. The proof is now included in~\cite{Romain} and builds on recent results by~\cite{CH20, CHS23, HMS23}.

		Let us move to the second item now. By definition of $h$ and $\varphi$, one has for any $y\in \mathcal H$,  
		$$h*\varphi(y) = \frac {1}{s^2}\sum_{z\in Q_s} \frac{1+r_{y,z}}{1+\|y-z\|^d} \lesssim \frac {1}{s^2}\sum_{z\in Q_s} \frac{1}{1+\|y-z\|^{d-1}}\lesssim 1,$$
		where the implicit constant in the last inequality is independent of $y$ and $s$; hence this proves the second item.

		Finally, we proceed with the last item of the lemma. Let $D_y =1+d(y,Q_s)$, and $r_y = 1+d(y,\partial \mathcal H)$. Observing that for any $z\in Q_s$, $r_{y,z} = d(y,\partial \mathcal H)$, we deduce that in fact 
		\begin{equation} \label{h*phi}
			h*\varphi(y)= \frac {r_y}{s^2}\sum_{z\in Q_s} \frac{1}{1+\|y-z\|^d}  \asymp \frac{r_y}{s^2}\cdot \frac {(s\wedge D_y)^{d-1}}{D_y^d},
		\end{equation}
		uniformly in $y\in \mathcal H$. Now, it amounts to estimating the sum 
		$$\sum_{z\in \mathcal H} h(y,z) \cdot (h*\varphi(z))^2.$$
		Assume first that $D_y\ge s$. Then  replacing $h*\varphi(z)$ by the previous bound yields
		\begin{align*}
			\sum_{z\in \mathcal  H} h(y,z) & \cdot (h*\varphi(z))^2 
			\lesssim \frac{r_y}{s^4} \cdot \sum_{z\in \mathcal H} \frac{r_z^2 (1+r_{z,y})\cdot (s\wedge D_z)^{2(d-1)}}{(1+\|y-z\|^d)\cdot D_z^{2d}}  \\
			& \lesssim \frac{r_y}{s^4} \cdot \Big\{ 
			\sum_{z : D_z \le s/2}\frac{r_z}{D_y^d}  + \sum_{z : s/2 \le D_z\le D_y/2} \frac{r_z^3 \cdot s^{2(d-1)}}{D_y^d\cdot D_z^{2d}} + \sum_{z : D_z \ge D_y/2} 
			\frac{r_z^2 \cdot s^{2(d-1)}}{ (1+\|y-z\|^{d-1})\cdot D_z^{2d}}\Big\} \\ 
			&\lesssim r_y \cdot \frac{s^{d-3}}{D_y^d} \asymp h*\varphi(y). 		\end{align*}
		On the other hand, if $D_y\le s$, then 
		\begin{align*}
			\sum_{z\in \mathcal H} h(y,z) \cdot (h*\varphi(z))^2 
			& \lesssim  \frac{r_y}{s^4} \cdot \sum_{z\in \mathcal H} \frac{r_z^2 (1+r_{z,y})\cdot (s\wedge D_z)^{2(d-1)}}{(1+\|y-z\|^d)\cdot D_z^{2d}}  \\
			& \lesssim \frac{r_y}{s^4} \cdot \Big(
			\sum_{z : D_z \le 2s}\frac{1}{1+\|y-z\|^{d-1}}  + \sum_{z : D_z\ge 2s} \frac{s^{2(d-1)}}{ D_z^{3d-3}}\Big) \\ 
			&\lesssim \frac{r_y}{s^3} \lesssim  \frac{r_y}{s^2 \cdot D_y} \asymp h*\varphi(y)
		\end{align*}
		and this finishes the proof.
	\end{proof}

	We can now conclude the proof of our main result. 
	
	\begin{proof}[\bf Proof of Theorem~\ref{thm.LD}] 
		We aim at applying Theorem~\ref{thm.generalLD} on the subspace $\Lambda = \mathcal H$, with  suitable multiples of the functions $h$ and $\varphi$. More precisely, we define $\widetilde h(x,y) = Ch(x,y)$ and $\widetilde {\varphi}(y) =  \varphi(y)/C$, with the constant $C$ from Lemma~\ref{lem.LD}. This lemma and Theorem~\ref{thm.generalLD} then ensure the existence of a constant $\lambda>0$, such that for any $x\in \mathcal H$, and any $s>0$, 
		$$\mathbb E\Big[ \exp\Big(\lambda \cdot \sum_{y\in \cC(x;\mathcal H)} \widetilde \varphi(y)\Big)\Big] \le 1+ \widetilde h*\widetilde \varphi(x),$$
		where we recall $ \cC(x;\mathcal H)= \{z\in \mathcal H : z\stackrel{\mathcal H}{\longleftrightarrow} x\}$, is the cluster of $x$ generated by bonds lying in $\mathcal H$. 
		Consequently, using that the exponential is equal to $1$ on the event $\{x\stackrel{\mathcal{H}}{\longleftrightarrow} Q_s\}^c$, we get that for any $x\in \mathcal H$, 
		$$\mathbb E\Big[ \exp\Big(\lambda \cdot \sum_{y\in  \cC(x;\mathcal H)} \widetilde \varphi(y)\Big) \cdot \1(  x\stackrel{\mathcal H}{\longleftrightarrow} Q_s)\Big] \le \widetilde h*\widetilde \varphi(x)+ \mathbb P\big(x\stackrel{\mathcal H}{\longleftrightarrow} Q_s\big).$$ 
		Using then Chebyshev's exponential inequality, we obtain for some (possibly smaller) constant $\lambda>0$, and any $t\ge 1$, 
		\begin{align}\label{eq:usefulfor8}
		\mathbb P\Big(| \cC(x;\mathcal H)\cap Q_s|\ge t\Big) 
		\le \Big( \widetilde h*\widetilde \varphi(x)+ \mathbb P\big(x\stackrel{\mathcal H}{\longleftrightarrow} Q_s\big)\Big)\cdot \exp(- \frac{\lambda \cdot t}{s^2}).
		\end{align}
		Now by~\eqref{h*phi}, one has when $s\le r_x$, 
		$$\widetilde h*\widetilde \varphi(x) \lesssim \frac {s^{d-3}} {D_x^{d-1}}, $$
		uniformly over $x\in\mathcal H$, and noting that for $x$ to be connected to $Q_s$, it must hold that $x$ is also connected to the boundary of the cube $B(x,r_x-1)$, we deduce from the one arm estimate~\eqref{onearm}, that 
		$$\mathbb P\big(x\stackrel{\mathcal H}{\longleftrightarrow} Q_s\big)\lesssim \frac 1{r_x^2},$$
		(uniformly over $s>0$ and $x\in \mathcal H$). Altogether this shows that when $s\le r_x$, 
		\begin{equation}\label{LD.hyperplane}
		\mathbb P\Big(| \cC(x;\mathcal H)\cap Q_s|\ge t\Big) 
		\lesssim \frac 1{r_x^2} \cdot \exp(- \frac{\lambda \cdot t}{s^2}). 
		\end{equation}
		To conclude, note first that one can always assume that $s\le r$. Then using~\eqref{LD.hyperplane}, and translation invariance, we get that for any $x\in \partial B(0,r)$, 
		$$\mathbb P\big(|\mathcal C_r(0)\cap  Q_s(x)|\ge t \big) = \mathbb P\big(|\cC(x;B(x,r)) \cap  B(0,s)\cap \partial B(x,r)|\ge t\big) \le d\cdot \mathbb P\big(|\cC(x;\mathcal H)\cap Q_s|\ge t/d\big),$$
		using also for the last inequality that $B(0,s)$ intersects at most $d$ faces of the cube $B(x,r)$. This concludes the proof of the theorem. 
	\end{proof}

	\begin{remark}\label{rem.simpleLD}
		\rm{Note that if one only uses the more basic bound $$\tau_{B(0,r)}(x,y)\le \tau(y-x)\lesssim \frac 1{1+ \|y-x\|^{d-2}},$$
			and define 
			$$\varphi(y) = \varepsilon \frac{\1(y\in Q_s(x))}{s^3},$$  one can check that for $\varepsilon$ small enough, the hypotheses of Theorem~\ref{thm.generalLD} are also  satisfied (with $h(x,y) = \tau(x,y)$). Indeed, one has in this case 
			$$h*\varphi(y) \asymp \frac 1{s^2} \cdot (\frac{s}{D_y}\wedge 1)^{d-2},$$
			and one can deduce using a similar computation as above, that an analogue of Lemma~\ref{lem.LD} holds as well. 
			This shows (without relying on the heavy machinery of~\cite{KN11,CH20,CHS23}) that for some constant $c>0$, for any $r,s,t\ge 1$, and any $x\in \partial B(0,r)$ and $z\in B(0,r)^c$,  
			\begin{align}\label{eq:badbound}
			\mathbb P(|\mathcal C_r(0)\cap Q_s(x)|\ge t) \le \exp(-c \frac{t}{s^3}) \quad \text{ and } \quad \pr{|\mathcal{C}(z;B(0,r)^c)\cap Q_s(x)|\geq t}\leq \exp(-c\frac{t}{s^3}). 
			\end{align}
			We note that in dimension $d>7$, this bound is sufficient to prove Claim \ref{claim:Tsloc} appearing in the next section, showing that most pioneer points are regular, with our definition of regular points (which is purely geometric, while originally in~\cite{KN11} it was defined in terms of a slightly less flexible conditional probability).}	
			\end{remark}
	
\begin{remark}\label{rem.circular} \rm{Let us stress that the proof of Theorem~\ref{thm.LD} only makes use of the {\it upper bound} on the half-space two-point function proved in~\cite{Romain}, whose proof does 
not rely on our theorem. It is only for the proof of the {\it lower bound} that~\cite{Romain} makes use of it.}
\end{remark}	


	\section{Regularity conditions}\label{sec.regpoints} 
	
	In this section we prove an important result, Proposition~\ref{thm:linegoodpoints}, that will be used in the next section for proving  Theorems~\ref{thm.IIC} and~\ref{thm.pcap.ball} (actually for the latter we will use a straightforward variant of it). The proposition states that most pioneer points on the boundary of a cube (those reached by the cluster of the origin from inside the cube) are regular, in the sense that the density of the cluster of the origin around them, both inside the cube and on its boundary, is not too high, at every scale. 
	A very similar result was proved in~\cite{KN11}, with the notable difference, that in~\cite{KN11} a more complicated notion of regular point was used as we explain below.

	We start with a number of definitions and notation used throughout this section and the next two. Recall also Definition~\ref{def.cluster}.

	\begin{definition}
		\rm{
			For $r>0$ we define the number of pioneer points $X_r$ as follows
			
			\[
			X_r=\left|\left\{x\in \partial \ball{0}{r}: 0\stackrel{ \ball{0}{r}}{\longleftrightarrow}   x    \right\}  \right|=|\mathcal C(0;B(0,r))\cap \partial B(0,r)|.	
			\]	
			For $x\in \partial \ball{0}{r}$ we define global and local density conditions. For $s>0$ we write
			\begin{align*}
				\cT_s(x) = & \{ |\cC(x;\ball{0}{r})\cap \ball{x}{s}
				| \leq s^4 (\log s)^7\} \\ &\cap \{|\cC(x;\ball{0}{r})\cap \ball{x}{s}\cap \partial \ball{0}{r} |\leq s^2 (\log s)^7  \}.
			\end{align*}
			We now define a local density condition as follows
			\begin{align*}
				\cT_s^\loc(x) =&\bigcap_{y\in \ball{x}{s}} \{|\cC(y; \ball{x}{s^d}\cap \ball{0}{r})\cap \ball{x}{s}|\leq s^4 (\log s)^4\} \\
				&\bigcap \bigcap_{y\in \ball{x}{s}\cap \partial \ball{0}{r}} \{|\cC(y;\ball{x}{s^d}\cap \ball{0}{r}) \cap \ball{x}{s}\cap \partial \ball{0}{r}|\leq s^2 (\log s)^4 \} \\
				&\bigcap \{ \exists \ \text{ at most }  (\log s)^3 \text{ disjoint paths from }\ball{x}{s} \text{ to } \partial \ball{x}{s^d} \text{ in } \ball{0}{r} \}.
			\end{align*}
		}
	\end{definition}
	
	\begin{definition}\label{def.Kreg}
		\rm{
			Let $x\in \partial \ball{0}{r}$. We call $x$ an $s$-(locally) good point if $\cT_s(x)$ (resp.\ $\cT_s^\loc(x)$) holds. Otherwise we call it an $s$-(locally) bad point. 
			We call $x$ a $K$-regular point if it is $s$-good for all $s\geq K$ and connected to the origin inside $B(0,r)$. Otherwise, if $x$ is connected to the origin in $B(0,r)$ but not $K$-regular, we call $x$ a $K$-irregular point. We denote by $X_r^{K-\reg}$ (resp.\ $X_r^{K-\irr}$) the number of $K$-regular (resp.\ irregular) points on $\partial \ball{0}{r}$. }
	\end{definition}

\begin{remark}\label{rem:difference}
		\rm{
We explain the difference between this notion of regularity and the one used in~\cite{KN11}. The first difference is that here we also control the number of pioneers on small balls on the boundary of the cube and not only in the interior of the cube. The second difference is that in~\cite{KN11} a point~$x$ is called $s$-good if the conditional probability of the event $\cT_s^\loc(x)$ (without the condition on the number of points on the small balls on the surface), given the cluster of $x$ inside the box, is smaller than what the large deviations of Aizenman and Newman~\cite{AN84} ensure. 
	}
	\end{remark}

	The next claim is essentially the same as Claim 4.2 from~\cite{KN11}. 
	
	\begin{claim}\label{cl:locgoodgood}
		For every $x\in \partial \ball{0}{r}$ and $s>0$ we have 
		\[
		\cT_s^\loc(x)\implies \cT_s(x).
		\]
	\end{claim}
	
	\begin{proof}[\bf Proof]

		We notice that for the cluster of $x$ created by bonds lying entirely in $\ball{0}{r}$ to be different to the one created from $\ball{x}{s^d}$, there must exist paths that connect $x$ to the complement of $\ball{x}{s^d}$. But any such connection must contain an open path from $\partial B(x, s^d)$ to $B(x,s)$. So since under $\cT_s^\loc(x)$ all clusters have size at most $s^4 (\log s )^4$ and there exist at most $(\log s)^3$ disjoint paths from $\ball{x}{s}$ to $\partial \ball{x}{s^d}$, it follows that the total size of the cluster of $x$ will be at most $s^4 (\log s)^7$. In the same way, we get that  the size of the cluster of $x$ on the surface will have size at most $s^2 (\log s)^7$. 
	\end{proof}

	In the following claim we show that most points on the boundary of a cube satisfy the local (hence the global) density condition, at any scale. 
	
	\begin{claim}\label{claim:Tsloc}
		There exists a positive constant $c$ so that for every $x\in \partial B(0,r)$ and $s\in (0,r)$ we have 
		\[
		\pr{\cT_s^\loc(x)} \geq 1 - \exp(-c(\log s)^4).
		\]
	\end{claim}
	
	\begin{proof}[\bf Proof]
		It was shown in~\cite[Proposition 5.1]{AN84} that for any $R>0$, and any $x\in B(0,R)$, one has 
		$$\mathbb P(|\mathcal C(x;B(0,R))|\ge t) \le (e/t)^{1/2} \cdot \exp(-ct/R^4),$$ for some constant $c>0$. This, and a union bound over all points $y\in B(x,s)$, allows to handle the probability that there are large clusters inside $B(x,s^d)$.  Likewise, Theorem~\ref{thm.LD} allows to control the size of the clusters on the boundary of the cube. As for the probability of having $(\log s)^3$ disjoint paths from $B(x,s)$ to $\partial B(x,s^d)$, we use BK inequality~\ref{BKineq}, together with Theorem~\ref{thm:onearm} and a union bound again over the starting points of the crossings on $\partial B(x,s)$, to show that it is of order at most $(s^{d-1} s^{-2d})^{(\log s)^3}=\exp(-d(\log s)^4)$, as desired. 
	\end{proof}

	\begin{remark}\label{rem:scubed}
		\rm{
			Let $\til{\cT}_s(x)$ and $\til{\cT_s}^\loc(x)$ be defined as the events ${\cT}_s(x)$ and ${\cT_s}^\loc(x)$ by replacing the bounds $s^2(\log s)^7$ and~$s^2(\log s)^4$ by $s^3(\log s)^7$ and~$s^3(\log s)^4$ respectively. Employing~\eqref{eq:badbound} from Remark~\ref{rem.simpleLD} then yields the same bound as in the statement of Claim~\ref{claim:Tsloc} and the statement and proof of Claim~\ref{cl:locgoodgood} both remain unchanged.
		}
	\end{remark}

	We can now state the main result of this section. 
	
	\begin{proposition}\label{thm:linegoodpoints}
		There exist two positive constants $c$ and $C$ so that for all $K$ sufficiently large and all positive $r$ and $M$, we have 
		\[
		\pr{X_r\geq M, X_r^{K-\reg}\leq \frac{X_r}{2} }\leq Cr^{d-1}\cdot \exp\left(-c\cdot (\log M)^4\right). 
		\]
	\end{proposition}

	Before proving the proposition above we state a result from~\cite{KN11} that will be useful in the proof. 
	
	\begin{lemma}{\rm{\cite[Lemma 1.1]{KN11}}}\label{lem:connectioninthebox}
		There exist two positive constants $c$ and $C$ so that for all $r$ and all $z\in \partial \ball{0}{r}$ we have 
		\[
		\pr{0\stackrel{\ball{0}{r}}{\longleftrightarrow}z} \geq c\cdot \exp(-C(\log r)^2).
		\]
	\end{lemma}

	\begin{proof}[\bf Proof of Proposition~\ref{thm:linegoodpoints}]
		Writing  $X_r^{s-\loc-\bad}$ for the number of $s$-locally bad points and using
		Claim~\ref{cl:locgoodgood} we have		
		\[
		X_r^{K-\irr} \leq \sum_{s\geq K} X_r^{s-\loc-\bad},
		\]
		and hence if $X_r^{K-\irr}\geq X_r/2$, then by taking $K$ large enough, this implies that there exists $s\geq K$ such that 
		\[
		X_r^{s-\loc-\bad} \geq \frac{X_r}{ s^2}.
		\]	
		So by a union bound we get 
		\[
		\pr{X_r\geq M,\  X_r^{K-\reg } \le \frac{X_r}{2}} \leq \sum_{s\geq K} \pr{X_r\geq M, \ X_r^{s-\loc-\bad} \geq \frac{X_r}{s^2}}.
		\]
		To control this last probability we will follow the same exploration procedure as in~\cite{KN11}. More precisely, for every $w\in \cU:=\{u\in \Z^d: u_i\in \{0,s^d\} \ \forall \ i=1,\ldots, d\}$, we subdivide the box $\ball{0}{r}$ into the boxes of the partition
		\[
		\mathcal B(w) = \{ \ball{0}{r}\cap \ball{z}{2s^d}: z\in w+4s^d \cdot \Z^d\}.
		\]
		We now fix $w$ and drop the dependence on $w$ from the notation to simplify the presentation. We fix an arbitrary order on the boxes of $\mathcal B(w)$, and  first explore the percolation configuration inside the boxes that do not touch $\partial \ball{0}{r}$. Using this configuration we pick the first box in the ordering that touches the boundary and which is connected to $0$ using the explored region, we reveal the percolation configuration inside the box and add it to the explored region.
		Then we pick the next unexplored box in the ordering which is connected to the origin in the explored region and continue until there are no more unexplored boxes connected to the origin. Note that for each explored box touching $\partial \ball{0}{r}$, the probability that this box contains a crossing from the cluster of the origin in the previously explored region up to
		a point $x\in \partial \ball{0}{r}$ such that $B(x,s^d)\cap B(0,r)$ is contained in the box, is lower bounded by $ce^{-C(\log s)^2}$, by Lemma~\ref{lem:connectioninthebox}. For simplicity if the box contains such crossing, we will just say that the box has a crossing to $\partial B(0,r)$.  
		Let $\cF_i$ be the percolation configuration of the explored region after adding the $i$-th box. 
		
		A box $q$ is called $s$-locally bad if it contains an $s$-locally bad vertex $x\in \partial \ball{0}{r}$ with $\ball{x}{s^d}\cap \ball{0}{r}\subset q$.  
		Let $q_1, q_2, \ldots$ be the sequence of explored boxes. Now let $\til{q}_1,\ldots$ be the subset of the collection of boxes that contain a crossing to $\partial\ball{0}{r}$.
		To be more rigorous, we define a sequence of stopping times by setting $\sigma_0=0$ and inductively we set for $i\geq 0$
		\[
		\sigma_{i+1}=\inf\{ k\ge \sigma_i+1: q_k \text{ has a crossing to } \partial \ball{0}{r}  \}. 
		\]
		If $\sigma_i<\infty$, we let $\til{q}_i=q_{\sigma_i}$  and $\til{\cF}_i=\cF_{\sigma_{i-1}}$, for every $i\geq 1$. 
		We then have 
		\begin{align*}
			\prcond{\til{q}_i \text{ is $s$-locally bad}}{\til{\cF}_i}{} =\sum_{\ell,k} \prcond{q_{\ell+k} \text{ is $s$-locally bad, } \sigma_{i-1}=\ell, \sigma_i-\sigma_{i-1}=k}{\til{\cF}_i}{}.
		\end{align*}
		We let $\tau_1,\ldots, \tau_k$ be the indices of the first,\ldots, $k$-th box that are explored after having revealed all the boxes up to time $\sigma_{i-1}$. We then have 
		\begin{align*}
			&\prcond{q_{\ell+k} \text{ is $s$-locally bad, }  \sigma_{i-1}=\ell, \sigma_i-\sigma_{i-1}=k}{\til{\cF}_i}{} \\
			&=\sum_{i_1,\ldots, i_k} \prcond{q_{\ell+k} \text{ is $s$-locally bad, } \sigma_{i-1}=\ell, \sigma_i-\sigma_{i-1}=k, \tau_1=i_1,\ldots, \tau_k=i_k}{\til{\cF}_i}{}\\
			&\le \sum_{i_1,\ldots, i_k}
			\prcond{q_{\ell +k} \text{ is $s$-locally bad, }\sigma_{i-1}=\ell, \sigma_i-\sigma_{i-1}>k-1, \tau_1=i_1,\ldots, \tau_k=i_k }{\til{\cF}_i}{}  \\
			&=\sum_{i_1,\ldots, i_k}\pr{q_{\ell +k} \text{ is $s$-locally bad}}\cdot \prcond{\sigma_{i-1}=\ell, \sigma_i-\sigma_{i-1}>k-1, \tau_1=i_1,\ldots, \tau_k=i_k}{\til{\cF}_i}{}, 
		\end{align*} 
		where for the last equality we use that the event that $q_{\ell +k}$ is $s$-locally bad is independent of the rest by definition. We therefore deduce
		\begin{align*}
			&\prcond{q_{\ell+k} \text{ is $s$-locally bad, }  \sigma_{i-1}=\ell, \sigma_i-\sigma_{i-1}=k}{\til{\cF}_i}{}  
			\\&\le  \pr{q_{\ell +k} \text{ is $s$-locally bad}}\cdot \prcond{\sigma_{i-1}=\ell, \sigma_i-\sigma_{i-1}>k-1}{\til{\cF}_i}{} .
		\end{align*}
		Now Lemma~\ref{lem:connectioninthebox} yields
		\begin{align*}
			\prcond{\sigma_{i-1}=\ell, \sigma_i-\sigma_{i-1}>k-1}{\til{\cF}_i}{}  \leq \prcond{\sigma_{i-1}=\ell}{\til{\cF}_i}{}  \cdot \left(1-c\cdot \exp(-C(\log s)^2) \right)^{k-1}.
		\end{align*}
		By a union bound and 
		Claim~\ref{claim:Tsloc}, we also get 
		\begin{align*}
			\pr{q_{\ell +k} \text{ is $s$-locally bad}}& \lesssim s^{d^2}\cdot\max_{x\in \partial B(0,r)}  \left(1-\pr{\cT_s^\loc(x)}\right) \lesssim \exp(-c(\log s)^4),
		\end{align*}
		for a positive constant $c$. 
		Hence, putting all pieces together and summing over all possible values of $k$ and $\ell$, gives
		\begin{align}\label{eq:uniformboundonbad}
			\prcond{\til{q}_i \text{ is $s$-locally bad}}{\til{\cF}_i}{}\lesssim \exp(C(\log s)^2-c(\log s)^4) \lesssim \exp(-c(\log s)^4/2).
		\end{align} 
		Now for $w\in \mathcal U$, we let $N(w)$ be the total number of explored boxes of the partition $\mathcal B(w)$ which have a crossing to $\partial B(0,r)$, and let $N_s(w)$ be the number of explored boxes which have a crossing to $\partial B(0,r)$ and are $s$-locally bad. 
		Note that for every $x\in \partial \ball{0}{r}$ there exists $w\in \cU$ and $q\in \mathcal B(w)$ such that $\ball{x}{s^d}\cap \ball{0}{r}\subseteq q$.
		We then have deterministically 
		\[
		X_r^{s-\loc-\bad}\le 2^ds^{d^2}\cdot \sum_{w\in \cU}N_s(w). 
		\]
		Using the above inequality and a union bound we then get 
		\begin{equation}\label{eq:binomial}
			\pr{X_r\geq M,  X_r^{s-\loc-\bad}\geq \frac{X_r}{s^2}} \leq 2^d\cdot \max_{w\in \cU} \pr{X_r\geq M, N_s(w)\geq \frac{X_r}{2^{2d}\cdot s^{d^2+2}}}. 
		\end{equation}
		Now, using~\eqref{eq:uniformboundonbad} we can upper bound $N_s(w)$ by a binomial random variable with parameters $N(w)$ and $\exp(-c(\log s)^4/2)$, independently of $N(w)$. Using moreover that 
		\[
		\sum_{w\in \cU}N(w)\leq X_r \leq 2^d s^{d^2} \sum_{w\in \cU}N(w),
		\]
		we obtain for any $w\in \cU$ that
		\begin{align*}
			&\pr{X_r\geq M, N_s(w)\geq \frac{X_r}{2^{2d}\cdot s^{d^2+2}}} \leq \pr{\sum_{u\in \cU}N(u)\geq \frac{M}{2^ds^{d^2}}, N_s(w)\geq \frac{\sum_{u\in \cU}N(u)}{2^{2d}\cdot s^{d^2+2}}} \\
			&\leq \pr{N(w)\geq \frac{M}{2^{2d}s^{d^2}}, N_s(w)\geq \frac{N(w)}{2^{2d}\cdot s^{d^2+2}}} + \pr{N(w)<\frac{M}{2^{2d}s^{d^2}}, N_s(w)\geq \frac{M}{2^{3d}s^{2d^2+2}}}.
		\end{align*}
		Using large deviations for Binomial random variables (e.g.\ Azuma's inequality) we can deduce that for $s\geq K$, and $K$ large enough
		\begin{align*}
			&\pr{N(w)\geq \frac{M}{2^{2d}s^{d^2}}, N_s(w)\geq \frac{N(w)}{2^{2d}\cdot s^{d^2+2}}} + \pr{N(w)<\frac{M}{2^{2d}s^{d^2}}, N_s(w)\geq \frac{M}{2^{3d}s^{2d^2+2}}} \\
			&\leq \E{\exp\left(-\kappa \cdot \frac{N(w)^2}{s^{2d^2+4}\cdot N(w)} \right)\cdot \1\left(N(w)\geq \frac{M}{2^{2d}s^{d^2}}\right)} \\
			&\quad \quad 
			\quad \quad \quad \quad \quad \quad \quad \quad + \E{\exp\left(-\kappa \cdot \frac{M^2}{s^{4d^2+4}\cdot N(w)} \right)\cdot \1\left(N(w)< \frac{M}{2^{2d}s^{d^2}}\right)}
			\\ &\leq \exp\left(-\kappa \cdot M \cdot s^{-3d^2-4} \right),
		\end{align*}
		for some positive constant $\kappa$.	
		Plugging this bound into~\eqref{eq:binomial} we deduce 
		\begin{align}\label{LDsloc}
			\pr{X_r\geq M,  X_r^{s-\loc-\bad}\geq \frac{X_r}{s^2}} \leq 2^d\cdot \exp\left(-\kappa \cdot M \cdot s^{-3d^2-4} \right).
		\end{align}
		Taking the sum over all $s\in [K,M^{1/(6d^2+8)}]$ we get 
		\begin{equation}\label{sumssmall}
			\sum_{s=K}^{M^{1/(6d^2+8)}}			\pr{X_r\geq M, \ X_r^{s-\loc-\bad}\geq \frac{X_r}{s^2}} \lesssim \exp(-\kappa\cdot \sqrt{M}). 	
		\end{equation}
		For $s\geq M^{1/(6d^2+8)}$ we  simply bound, using Claim~\ref{claim:Tsloc},  
		\begin{align*}
			\pr{X_r\geq M, X_r^{s-\loc-\bad} \geq \frac{X_r}{s^2}} & \leq \pr{X_r^{s-\loc-\bad}\geq 1} \lesssim  r^{d-1} \max_{x\in \partial B(0,r)} \pr{(\cT_s^\loc(x))^c } \\
			& \lesssim r^{d-1}\cdot \exp(-c(\log s)^4).
		\end{align*}
		Finally summing over  $s\geq M^{1/(6d^2+8)}$, and using~\eqref{sumssmall}, yields the desired result. 
	\end{proof}

	\begin{remark}\label{rem:newdefscubed}
		\rm{
		Using Remark~\ref{rem:scubed} we see that the statement and proof of Proposition~\ref{thm:linegoodpoints} remain unchanged if we consider the events $\til{\cT}_s$ and $\til{\cT}_s^\loc$ instead of $\cT_s$ and $\cT_s^\loc$ respectively.
		}
	\end{remark}

\begin{remark}\label{rem:expreg}
\rm{A simpler statement than Proposition~\ref{thm:linegoodpoints}, which can be proved using the same exploration as in the proof above, is that for $K$ large enough, one has 
\begin{equation}\label{eq.expreg}
\mathbb E[X_r^{K-\reg}]\ge \frac 12 \cdot \mathbb E[X_r]. 
\end{equation}
Indeed, to see this, it suffices to observe that in the exploration procedure, for each explored box, the probability that
it is $s$-locally bad conditionally on the fact that there is a crossing in the box, is exponentially small. Hence in total the expected number of $s$-locally bad pioneers is exponentially small (in $s$) compared to the expected number of pioneers. Summing over $s\ge K$, gives~\eqref{eq.expreg}. This remark will be used in Section~\ref{sec:tile}. 
}
\end{remark}

	
	\section{One arm exponent and hitting probabilities for the IIC}\label{sec:hitting}

	In this section we prove Theorem~\ref{thm.IIC} and give an alternative proof of the main ingredient of the proof of the one arm exponent of~\cite{KN11} for dimensions $d\geq 8$ under~\eqref{twopoint}.

	\subsection{Preliminaries}
	
We start by proving a result, which states that the $(d-4)$-capacity of a sufficiently sparse set is comparable to its volume. 
	
	\begin{claim}\label{cl:lowerboundoncapacity}
		Let $C$ be a positive constant, $d\geq 5$ and $\alpha<d-4$. There exists $c>0$, so that if  $A\subseteq \Z^d$ is a finite set satisfying that for all $x\in A$ we have 
		\[
		|\ball{x}{s}\cap A|\leq C s^{\alpha}, \quad  \forall  s>0,
		\]
		then 
		\begin{align}\label{eq:veryuseful}
		C_{d-4}(A) \geq c\cdot |A| \quad \text{ and } \quad \sum_{x,y\in A} G(x,y)\lesssim |A|.
		\end{align}
	\end{claim}

	\begin{proof}[\bf Proof]
		
		By taking $\nu$ to be the uniform measure on $A$ in the definition~\eqref{var.C4} of $C_{d-4}$, we get, with $G(x,y)=\frac 1{(1+\|x-y\|)^{d-4}}$,  
		\begin{align*}
			C_{d-4}(A)\geq \frac{|A|^2}{\sum_{x,y\in A} G(x,y)}.
		\end{align*}
		Hence, to prove the claim it suffices to show that 
		\[
		\sum_{x,y\in A} G(x,y) \lesssim |A|. 
		\]
		Setting $B(x,1/2)=\{x\}$, we now have 
		\begin{align*}
			\sum_{x,y\in A} G(x,y) &=\sum_{x\in A} G(x,x) + \sum_{x\in A} \sum_{t=0}^{\infty} \sum_{y\in A\cap  \ball{x}{2^t}\setminus \ball{x}{2^{t-1}}} G(x,y) \\&\lesssim |A|+\sum_{x\in A} \sum_{t=0}^{\infty} \frac{|\ball{x}{2^t}\cap A| }{2^{t(d-4)}} 
			\lesssim |A| + \sum_{x\in A}\sum_{t=0}^{\infty} \frac{2^{\alpha t}}{2^{t(d-4)}} \asymp |A|,
		\end{align*}
		and this completes the proof, recalling that $d\geq 5$ and  $\alpha<d-4$, by assumption. 
	\end{proof}
	
	Another standard and useful fact is the following. 
	
	\begin{claim}\label{claim.G}
		Let $G(x,y) = \frac 1{(1+\|x-y\|)^{d-4}}$. Then 
		$$\sum_{u\in \mathbb Z^d} \tau(x-u)\tau(u-y) \asymp G(x,y). $$ 
	\end{claim}
	\begin{proof}[\bf Proof] The proof immediately follows from~\eqref{twopoint} and a standard computation. 
	\end{proof}
	
	We now prove a lower bound on the probability to hit a set, which also gives  Proposition~\ref{prop:lowercap} as a corollary, as we explain immediately after the proof of the lemma below. 
	\begin{lemma}\label{lem:hitdiameterdistance}
		For any positive constant $c$ there exists $c'>0$, such that the following holds.
		Let $A\subseteq \Z^d$ be a finite set and let $x\in \Z^d$ satisfy $ d(x,A)\ge c \cdot \max_{a\in A} \|a\|$. Then 
		\[
		\pr{x\longleftrightarrow A} \ge c'\cdot  (d(x,A))^{2-d} \cdot C_{d-4}(A).
		\]
	\end{lemma}
	
	\begin{proof}[\bf Proof]
		Let $L=d(x,A)$ and let $\nu$ be a probability measure supported on $A$. Define 
		\[
		Z=\sum_{a\in A} \nu(a) \frac{\1(x\longleftrightarrow a)}{\tau(x-a)}.
		\]
		By Cauchy-Schwarz inequality we have 
		\begin{align*}
			\pr{x\stackrel{}{\longleftrightarrow}A}\geq \pr{Z>0} \geq \frac{(\E{Z})^2}{\E{Z^2}}.
		\end{align*}
		For the first moment of $Z$ we get
		\[
		\E{Z} = \sum_{a\in A} \nu(a) \frac{\tau(x-a)}{\tau(x-a)} = 1.
		\]
		Since $d(x,A)\geq c\cdot \max_{a\in A} \|a\| $, it follows that $\|x-a\|\asymp \|x-b\| \asymp L$, and hence for the second moment of $Z$ we have, using~\eqref{twopoint}, 
		\begin{align*}
			\E{Z^2} =\sum_{a,b\in A} \nu(a) \nu(b)\frac{ \pr{a\longleftrightarrow x,a\longleftrightarrow b}}{\tau(x-a)\tau(x-b)} \asymp L^{2d-4}\cdot  \sum_{a,b\in A} \nu(a) \nu(b) \pr{a\longleftrightarrow x, a\longleftrightarrow b}. 
		\end{align*}
		Let $C$ be a sufficiently large constant. By the BK inequality~\eqref{BKineq} we get, using Claim~\ref{claim.G}, 
		\begin{align*}
			& \pr{a\longleftrightarrow x, x\longleftrightarrow b}\leq \sum_{z} \tau(a-z) \tau(z-x) \tau(z-b)\\
			& = \sum_{\|z-x\|\leq \frac{L}{C}}\tau(a-z) \tau(z-x) \tau(z-b) + \sum_{\|z-x\|>\frac{L}{C}} \tau(a-z) \tau(z-x) \tau(z-b)  \\
			&\asymp L^{4-2d} \cdot \sum_{\|z-x\|\leq \frac{L}{C}}\tau(z-x) + L^{2-d} \cdot G(a,b) \asymp L^{6-2d} + L^{2-d} G(a,b). 
		\end{align*}
		Substituting this above gives
		\[
		\E{Z^2}\lesssim L^{2}+ L^{d-2}\cdot \sum_{a,b\in A} \nu(a) \nu(b) G(a,b),
		\]
		and hence combining everything this implies
		\[
		\pr{x\longleftrightarrow A}\geq 	\pr{Z>0}\geq \frac{1}{L^2 + L^{d-2} \cdot \sum_{a,b\in A} \nu(a) \nu(b) G(a,b)}.
		\]
		Since 
		\[
		\sum_{a,b \in A} \nu(a) \nu(b) G(a,b) \gtrsim L^{4-d},
		\]
		we obtain
		\[
		\pr{Z>0} \gtrsim \frac{L^{2-d}}{\sum_{a,b\in A}\nu(a)\nu(b)G(a,b)}.
		\]
		Optimising over the choice of the probability measure $\nu$ proves the claim. 
	\end{proof}

	\begin{proof}[\bf Proof of Proposition~\ref{prop:lowercap}]
		The proof follows immediately from Lemma~\ref{lem:hitdiameterdistance} and Theorem~\ref{thm.defpcap}.
	\end{proof}
	
	We next need to introduce some more notation and definitions. 
	First, recall the notation and definitions of Section~\ref{sec.regpoints}. Then, given $r\ge 1$ and any  connected subgraph $B\subseteq B(0,r)$ containing the origin, we define the set of $K$-regular points of $B$, as the points $x\in B\cap \partial B(0,r)$, such that for all $s\ge K$, 
	$$|B\cap B(x,s)|\le s^4(\log s)^7, \quad \text{and}\quad |B\cap B(x,s)\cap \partial B(0,r)|\le s^2 (\log s)^7.$$ 
	Note that on the event $\mathcal C(0;B(0,r))= B$, these are exactly the set of $K$-regular points as given in Definition~\ref{def.Kreg}. Among them we fix arbitrarily a subset of maximal cardinality whose points are all within distance at least $2K$ from each other, and call it $\Reg{B}$. Now for each point $x\in \Reg{B}$, we consider the line segment of length $K$ emanating from $x$ outside $B(0,r)$ orthogonally to the boundary of the cube (or if $x$ belongs to more than one face, orthogonally to one of them chosen arbitrarily), and call it $L_x$. We also call (with a slight abuse of notation) $x'$ the endpoint of this line, which belongs to $\partial B(0,r+K)$ by definition.  
	For each $x\in \Reg{B}$, we consider the maximal line segment starting from $x$ and contained in $L_x$ formed by open edges, and denote by $L(B)$ the union of all these line segments (note that this is a random set) and we are abusing notation, since $L(B)$ depends also on the configuration outside of $B$. A point $x\in \Reg{B}$ is said to be $K$-line-good, if all the edges of $L_x$ are open (or equivalently if $x'\in L(B)$). We write $\text{Reg}'(B)$ for the set of such $x'$, and let $X_r^{K-\line}(B)=|\text{Reg}'(B)|$ be the number of $K$-line-good points of $B$.  
	We also let 
	$$X_r(B) = B\cap \partial B(0,r),$$
	be the set of pioneer points corresponding to $B$. We now define the extended cluster of the origin as $$\ce{\ball{0}{r}} = \mathcal C(0;B(0,r))\cup L\big(\mathcal C(0;B(0,r))\big),$$   
	and write $X_r^{K-\line}$ for the number of $K$-line-good points of $\mathcal C(0;B(0,r))$. 
	We shall need the following simple fact. 
	\begin{claim} 
		\label{claim:Klinegood} For all $K$, there exists a constant $\varepsilon(K)>0$, such that for all positive $r$ and $M$, 
		$$\mathbb P\Big(X_r^{K-\reg} \ge M, X_r^{K-\line}\le \epsilon(K) \cdot X_r^{K-\reg}  \Big)\lesssim  \exp(-\epsilon(K)\cdot M).$$
	\end{claim}
	\begin{proof}[\bf Proof]
		Note that from the definition of 
		$\Reg{B}$, one can see that 
		$$N:=|\Reg{\mathcal C(0;B(0,r))}|\gtrsim  \frac{X_r^{K-\reg}}{K^d}.$$  
		Now conditionally on $N$, the number of $K$-line-good points is a binomial random variable with parameters $N$ and $p_c^K$. Hence the result follows immediately from Azuma's inequality. 
	\end{proof}

	\begin{remark}\label{rem:scubedlinegood}
		\rm{
		Note that if we use the events $\til{\cT}_s$ and $\til{\cT}_s^\loc$ instead of $\cT_s$ and $\cT_s^\loc$ in the definition of regular points, then the statement and proof of Claim~\ref{claim:Klinegood} remain unchanged.
		}
	\end{remark}

		\subsection{Another proof of~\cite[Theorem~2]{KN11}}\label{sec:newproof}

	In this section we present a different proof of Theorem~2 from~\cite{KN11} for $d\geq 8$ and under~\eqref{twopoint}. This is the main ingredient in the proof of the one-arm exponent in~\cite{KN11} (established there as long as~\eqref{twopoint} holds including $d=7$). 
	In this section we use the definitions 
	$\til{\cT}_s$ and $\til{\cT}_s^\loc$ from Remark~\ref{rem:scubed} when considering regular points.

	Define
	\[	 
	A_{r,L} = \left|\{x\in  \ball{0}{r+L}\setminus \ball{0}{r}: z\longleftrightarrow  x\}\right|.
	\]
	\begin{theorem}\label{thm:3}\rm{(\cite[Theorem 2]{KN11})}
		There exists a constant $c>0$ so that for all $\alpha>0$, all $r$ sufficiently large and for all $L\geq r^\alpha$
		\[
		\pr{X_{r}\geq L^2 \  \text{ and } \ A_{r,L}\leq cL^4} \leq (1-c)\cdot \pr{0\longleftrightarrow  \partial B(0,r)}.
		\]
	\end{theorem}

		We now briefly explain the proof of Theorem~\ref{thm:3} in~\cite{KN11} and highlight the difference with our approach.  In~\cite{KN11}, the authors define a random variable $Y$ 
		which counts the number of vertices in the annulus that are connected to a single regular point on $\partial \ball{0}{r}$. This variable $Y$ is then a lower bound on the variable $A_{r,L}$ and hence the probability appearing in the statement of Theorem~\ref{thm:3} can be upper bounded by the same probability with $A_{r,L}$ replaced by $Y$. The next steps in~\cite{KN11} consist in establishing first and second moments of $Y$ conditional on the event that there are~$M$ regular points. The proof is then concluded using the large deviations results on the number of regular points together with the Payley-Zygmund inequality applied to the variable $Y$. In this section, instead of working with $Y$, we define another random variable $Z_{r,L}$ appearing below which also serves as a lower bound on $A_{r,L}$. Then as in~\cite{KN11} we bound the first and second moments of $Z_{r,L}$ on the event that there are~$M$ regular points. Bounding the first moment of $Y$ in~\cite{KN11} is quite involved and requires a different definition of regularity to the one we use here, while in the case of $Z_{r,L}$, it becomes much shorter using the idea of line-good points as we present below. The second moments of $Y$ and $Z_{r,L}$ are computed through a very similar case by case analysis.

	To this end for $L\ge 10K$ we define  $\mathcal S_L = B(0,r+L) \setminus B(0,r+L/2)$, and
	\[
	Z_{r,L} = \sum_{y\in \mathcal S_L} \1\big(y\stackrel{\text{off } \ce{B(0,r)}}{\longleftrightarrow}\text{Reg}'(\cC(0;\ball{0}{r}))\cap B(y,L)\big).
	\]

	\begin{lemma}\label{lem:firstmomentofy1}
		There exists $c>0$, such that for all $K$ sufficiently large, all $L\geq 10K$, $M\ge 1$, and $r\ge 2L$, we have 
		\[
		\E{Z_{r,L}\cdot \1(X_r^{K-\line}=M)} \ge c L^2 \cdot M\cdot 
		\pr{X_r^{K-\line}=M}.
		\]	
	\end{lemma}
	
	\begin{proof}[\bf Proof]
		Given $H$ a possible realization of an extended cluster $\ce{B(0,r)}$, we write 
		$$B= H\cap B(0,r). $$ 
		The n 
		\begin{align}\label{eq:ZrL.1}
			\nonumber		&\E{Z_{r,L}\cdot \1(X_r^{K-\line}=M)} = \sum_{H: |\text{Reg}'(B)|=M} \E{Z_{r,L}\cdot \1(\cez{\ball{0}{r}}=H)}\\
			\nonumber			& = \sum_{H: |\text{Reg}'(B)|=M} \E{\1(\cez{\ball{0}{r}}=H) \sum_{y\in \mathcal S_L} \1(y\stackrel{\text{off } H}{\longleftrightarrow}\text{Reg}'(B)\cap B(y,L) )}\\
			&=\sum_{H: |\text{Reg}'(B)|=M} \sum_{y\in \mathcal S_L} \pr{\cez{\ball{0}{r}}=H} \pr{y\stackrel{\text{off } H}{\longleftrightarrow}\text{Reg}'(B) \cap B(y,L)},
		\end{align}
		where  the two events above decorrelate, since they depend on disjoint sets of edges. 
		We next prove that there exists a universal positive constant $c$ so that for any $y\in \mathcal S_L$, 
		\begin{align}\label{eq:goalforytohittl.onearm}
			\pr{y\stackrel{\text{off } H}{\longleftrightarrow}\text{Reg}'(B)\cap B(y,L)} \geq c\cdot  L^{2-d} \cdot |\text{Reg}'(B)\cap B(y,L)|.
		\end{align}
		To this end assign an ordering to the points of $\text{Reg}'(B)$, and on the event $\{y\stackrel{\text{off } H}{\longleftrightarrow} \text{Reg}'(B)\}$, let~$Y$ be the first point in the ordering such that $y$ is connected to the boundary of the ball of radius $K$ around this point off $H$. For $x\in \text{Reg}'(B)$, let $L_x$ be the line segment of length $K$ from $x$ to $\partial B(0,r)$ (which by definition is part of $H$), and let $$H^* = H\setminus (\cup_{x\in \text{Reg}'(B)}L_x).$$  
		We claim that
		\begin{align}
			\label{eq:firstballhit.onearm}
			\begin{split}
				&\pr{y\stackrel{\text{off } H}{\longleftrightarrow} \text{Reg}'(B)\cap B(y,L) } \\
				&\geq  \sum_{x\in \text{Reg}'(B)\cap B(y,L)} \pr{Y=x, y\stackrel{\text{off } H^*}{\longleftrightarrow} \ball{x}{K},  \text{ all edges in } \ball{x}{K} \text{ are open}}.
			\end{split}
		\end{align}
		Indeed, notice that on the event in the probability above, there exists an open path connecting $y$ to $B(x,K)$ off $H$, since the lines $L_x$ are part of the balls $B(x,K)$. Let~$z$ be the vertex on $\partial B(x,K)$ that $y$ is connected to. Since all edges of $B(y,K)$ are open, this implies that there is an open path inside $B(y,K)$ connecting $z$ to $x$ without intersecting $L_x$, and hence this means that $y\stackrel{\text{off } H}{\longleftrightarrow} \text{Reg}'(B)\cap B(y,L)$. 
		
		Since the two events appearing in the probability on the right hand side of~\eqref{eq:firstballhit.onearm} are independent, they factorise, and hence we get with $c(K)=p_c^{d(2K+1)^{d}}$,  
		\begin{align*}
			\pr{y\stackrel{\text{off }H}{\longleftrightarrow} \text{Reg}'(B)\cap B(y,L)}&\geq  c(K)\cdot \sum_{x\in \text{Reg}'(B)\cap B(y,L)} \pr{Y=x, y\stackrel{\text{off } H^*}{\longleftrightarrow}  \ball{x}{K}}\\
			& = c(K) \cdot \mathbb P\Big(y\stackrel{\text{off }H^*}{\longleftrightarrow} \bigcup_{x\in \text{Reg}'(B)\cap B(y,L)}B(x,K)\Big)
			\\\ &\geq c(K)\cdot \pr{y\stackrel{\text{off }H^*}{\longleftrightarrow} \text{Reg}'(B)\cap B(y,L)},
		\end{align*}
		where the last inequality follows from the fact that to reach $\text{Reg}'(B)\cap B(y,L)$, the cluster of $y$ must necessarily first reach $\cup_{x\in \text{Reg}'(B)\cap B(y,L)}\partial B(x,K)$.
		It remains to bound the last probability above. First we write 
		\begin{equation}\label{eq:hitofftildb.onearm}
			\pr{y\stackrel{\text{off }H^*}{\longleftrightarrow} \text{Reg}'(B)\cap B(y,L)} = \pr{y\longleftrightarrow \text{Reg}'(B)\cap B(y,L)} -  \pr{y\stackrel{\text{via }H^*}{\longleftrightarrow} \text{Reg}'(B)\cap B(y,L)}, 
		\end{equation}
		where ``via $H^*$" means using a path of open edges intersecting $H^*$.  
		Next by the BK inequality~\eqref{BKineq} we get,
		\begin{align}\label{eq:boundviabtilde.onearm}
			\nonumber	&	\pr{ y\stackrel{\text{via } H^*}{\longleftrightarrow} \text{Reg}'(B)\cap B(y,L)} \leq \sum_{b\in H^*} \pr{y\longleftrightarrow b} \cdot \pr{b\longleftrightarrow \text{Reg}'(B)\cap B(y,L)}\\
			& \lesssim L^{2-d} \cdot \sum_{b\in H^*} \pr{b\longleftrightarrow \text{Reg}'(B)\cap B(y,L)} 	\lesssim L^{2-d} \cdot  \sum_{b\in H^*}\sum_{x'\in \text{Reg}'(B)\cap B(y,L)} \tau(x'-b), 
		\end{align}
		where for the second inequality we used that $\|y-b\|\asymp L$ and the last inequality follows by a union bound. 
		We set for $t>0$
		\[
		B_t(x') = H^*\cap \ball{x'}{2^t}\setminus \ball{x'}{2^{t-1}}.
		\]
		Then the above sum can be split as follows, using~\eqref{twopoint}, 
		\begin{align*}
			\sum_{x'\in \text{Reg}'(B)} \sum_{b\in H^*} \tau(x'-b) \asymp \sum_{x'\in \text{Reg}'(B)}\sum_{t=0}^{\log_2 (2r)} \frac{|B_t(x')|}{2^{t(d-2)}}.
		\end{align*}
		Recalling that $H^*$ is at distance $K$ from $x'$, we see that 
		$B_t(x')=\emptyset$, for $t< \log_2(K)$.
		For $t\geq \log_2(K)$ writing $x\in \Reg{B}$ for the point on $\partial \ball{0}{r}$
		corresponding to $x'$ and using the density condition for $x$, i.e.\ that the event $\cT_s(x)$ holds, since $x\in \Reg{B}$, we obtain (taking also into account points in $L(B)$ which may be part of $H^*$), 
		\[
		|B_t(x')| \leq  |\ball{x}{2^{t+1}}\cap H^* | \leq (2^{t+1})^4(\log(2^{t+1}))^7+K(2^{t+1})^3 (\log(2^{t+1}))^7 \asymp t^7\cdot (2^{4t} + K \cdot 2^{3t}), 
		\]
		Inserting this bound in the previous sum, this yields 
		\begin{align*}
			\sum_{x'\in \text{Reg}'(B)} \sum_{b\in H^*} \tau(x'-b) & \asymp \sum_{x'\in \text{Reg}'(B)}\sum_{t=\log_2(K)}^{\log_2(2r)} \frac{|B_t(x')|}{2^{t(d-2)}}	\lesssim 
			\sum_{x'\in \text{Reg}'(B)}\sum_{t=\log_2(K)}^{\infty} \frac{t^7\cdot (2^{4t} + K \cdot 2^{3t})}{2^{t(d-2)}}\\ & \lesssim \frac{1}{\sqrt K} |\text{Reg}'(B)|, 
		\end{align*}
		where for the last inequality we use that $d>6$. 
		Plugging this bound into~\eqref{eq:boundviabtilde.onearm}
		we deduce 
		\begin{align}\label{eq:hitviatildbsmall.onearm}
			\pr{y\stackrel{\text{via }H^*}{\longleftrightarrow} \text{Reg}'(B)\cap B(y,L)}\lesssim \frac{1}{\sqrt K} \cdot L^{2-d} \cdot |\text{Reg}'(B)\cap B(y,L)|.
		\end{align}
		On the other hand, we get from Lemma~\ref{lem:hitdiameterdistance}
		\begin{align*}
			\pr{ y{\longleftrightarrow} \text{Reg}'(B)\cap B(y,L)}\gtrsim L^{2-d} \cdot \cpc{d-4}{\text{Reg}'(B)\cap B(y,L)}\gtrsim L^{2-d}\cdot |\text{Reg}'(B)\cap B(y,L)|,
		\end{align*}
		where for the second inequality we used Claim~\ref{cl:lowerboundoncapacity}, since the set $\text{Reg}'(B)$ satisfies the density condition of the statement of the claim by definition and $d\geq 8$.
		Plugging the above together with~\eqref{eq:hitviatildbsmall.onearm} into~\eqref{eq:hitofftildb.onearm} and taking $K$ sufficiently large gives
		\begin{equation}\label{eq:goalforw2.onearm}
			\pr{y\stackrel{\text{off }H^*}{\longleftrightarrow} \text{Reg}'(B)\cap B(y,L)} \gtrsim L^{2-d}\cdot |\text{Reg}'(B)\cap B(y,L)|.
		\end{equation}		
		Then plugging~\eqref{eq:goalforytohittl.onearm} into~\eqref{eq:ZrL.1} we get 
		\begin{align*}	\E{Z_{r,L}\cdot \1(X_r^{K-\line}=M)}   
			&\gtrsim \sum_{H: |\text{Reg}'(B)|=M}  \pr{\ce{\ball{0}{r}^c}=H}\cdot L^{2-d} \cdot\sum_{y\in \mathcal S_L} |\text{Reg}'(B)\cap B(y,L)|\\
			&\gtrsim L^{2-d} \cdot \sum_{H: |\text{Reg}'(B)|=M} \pr{\ce{\ball{0}{r}^c}=H}\cdot \sum_{x\in \text{Reg}'(B)} |B(x,L)\cap \mathcal S_L| \\
			&\gtrsim L^2\cdot M \cdot \pr{X_r^{K-\line}= M},
		\end{align*}
		using a sum inversion for the second inequality. This finishes the proof. 
	\end{proof}

	\begin{lemma}\label{lem:secondmomentofy1}
		For any positive constant $c$, there exists $C>0$, such that for all $L\ge 10K$, $M\ge cL^2$, and $r\ge 2L$, we have 
		\begin{align*}
			\E{Z_{r,L}^2\cdot \1(X_r^{K-\line}=M)}\leq C\cdot  M^2\cdot L^4\cdot \pr{X_r^{K-\line}=M}.
		\end{align*}
	\end{lemma}

	\begin{proof}[\bf Proof]
		
		Let $H$ be a configuration such that $\pr{\ce{\ball{0}{r}^c}=H}>0$ and $|\text{Reg}'(B)|=M$ with~$B=H\cap \ball{0}{r}^c$.	We then have 
		\begin{align}\label{eq:secondmomentofy}
			\begin{split}
				&	\E{Z_{r,L}^2\cdot \1(\ce{\ball{0}{r}}=H)} \\
				&	= \sum_{x_1,x_2 \in \text{Reg}'(B)}\sum_{\substack{y_1\in \mathcal S_L \cap B(x_1,L) \\y_2\in \mathcal S_L \cap B(x_2,L) }}  \mathbb P\Big(\ce{\ball{0}{r}}=H, y_1\stackrel{\text{off } H}{\longleftrightarrow} x_1, y_2\stackrel{\text{off } H}{\longleftrightarrow} x_2\Big) \\
				&\leq \pr{\ce {\ball{0}{r}}=H}\sum_{x_1,x_2 \in \text{Reg}'(B)}  \sum_{\substack{y_1\in \mathcal S_L \cap B(x_1,L) \\y_2\in \mathcal S_L \cap B(x_2,L) }}  \pr{y_1\longleftrightarrow x_1, y_2\longleftrightarrow x_2},
			\end{split}
		\end{align}
		where for the last inequality first we used the independence of the two events, as they depend on disjoint sets of edges, and then we upper bounded the probabilities by also allowing the paths to use edges of $H$. 
		
		We first treat the case $x_1=x_2=x$ and $y_1=y_2=y$. Then the sum above becomes
		\begin{align*}
 \sum_{x\in \text{Reg}'(B)}  \sum_{\substack{y\in \mathcal S_L \cap B(x,L)  }}  \tau(x-y) \asymp L^2 \cdot M.
		\end{align*}
		Next we consider the case $x_1=x_2$ and $y_1\neq y_2$. Then using the BK inequality and 	that uniformly over $u,x\in \mathbb Z^d$, one has 
		\begin{align}\label{eq:uniformoverz}
			\sum_{y\in B(x,L)}\tau(y-u) \lesssim L^2,
		\end{align}
		the sum above becomes upper bounded by 
		\begin{align*}
		& \sum_{x\in \text{Reg}'(B)}  \sum_{\substack{y_1, y_2\in \mathcal S_L \cap B(x,L)  }} \sum_z \tau(x-z)\tau(z-y_1)\tau(z-y_2) \\ &\lesssim L^2\cdot  \sum_{x\in \text{Reg}'(B)}\sum_{y_1\in \mathcal S_L \cap B(x,L) }\sum_z \tau(x-z)\tau(z-y_1)
		 \asymp L^2\cdot M\cdot L^4 = L^6\cdot M \lesssim  M^2\cdot L^4
		\end{align*}
	using our assumption on $M$. If $x_1\neq x_2$ and $y_1=y_2$, then the same argument gives an upper bound of order~$L^2\cdot M^2$. Therefore, for all these cases the sum appearing in~\eqref{eq:secondmomentofy} is upper bounded by~$L^4\cdot M^2$. 
	
	We now concentrate on the case when $x_1\neq x_2$ and $y_1\neq y_2$. We notice that if $x_1$ is connected to $y_1$ and $x_2$ is connected to $y_2$, then either there exist two disjoint paths connecting them, or there exists $u$ and $v$, such that the following event holds: 
		\begin{align}\label{disjointpaths}
			\begin{split} &\{x_1\longleftrightarrow u\}\circ \{u\longleftrightarrow v\} \circ \{v\longleftrightarrow y_1\}\circ \{x_2\longleftrightarrow u\}\circ \{v\longleftrightarrow y_2\}\\
				& \cup \{x_1\longleftrightarrow u\}\circ \{u\longleftrightarrow v\} \circ \{v\longleftrightarrow y_1\}\circ \{x_2\longleftrightarrow v\}\circ \{u\longleftrightarrow y_2\}.
			\end{split}
		\end{align}
		Indeed, to see this, consider two simple paths connecting $x_1$ and $y_1$, and $x_2$ and $y_2$ respectively, then let $u$ be the first point of intersection of the first path with the second one, and $v$ the last intersection point (where first and last intersection points are understood from the point of view of the first path say). Then using that the paths from $x_1$ to $u$ and from $v$ to $y_1$ are disjoint from the path from $x_2$ to $y_2$, we can infer~\eqref{disjointpaths} (where the two events correspond to whether this last path first visit $u$ and then $v$ or the opposite). Hence we deduce, using BK inequality~\eqref{BKineq},  
		\begin{align*}
			\pr{y_1\longleftrightarrow x_1, y_2\longleftrightarrow x_2 } & \leq \tau(y_1- x_1)\tau(y_2-x_2) + \sum_{u,v}  \tau(x_1-u)\tau(u-v)\tau(v-y_1)\tau(x_2-u)\tau(v-y_2)\\
			&  \quad +\sum_{u,v} \tau(x_1-u)\tau(u-v)\tau(v-y_1)\tau(x_2-v)\tau(u-y_2). 
		\end{align*}
		Note that the first term on the right-hand side above is of order $L^{4-2d}$, by~\eqref{twopoint}, since by definition $x_1$ and $y_1$ are at distance order $L$ from each other, and similarly for $x_2$ and $y_2$. Using again~\eqref{eq:uniformoverz} we get
		\begin{align*}
			& \sum_{\substack{y_1\in \mathcal S_L \cap B(x_1,L) \\y_2\in \mathcal S_L \cap B(x_2,L) }} \pr{y_1\longleftrightarrow x_1, y_2\longleftrightarrow x_2 }\\ 
			& \lesssim L^4  + L^2 \sum_{u,v} \sum_{y \in \mathcal S_L\cap B(x_2,L)} \tau(x_1-u) \tau(u-v)\tau(x_2-u)\tau(v-y) + L^4 \sum_{u,v} \tau(x_1-u)\tau(u-v) \tau(x_2-v). 
		\end{align*}		
		Using next Claim~\ref{claim.G}, we obtain with the function $G$ as in the claim,  
		\begin{align*}
			& \sum_{\substack{y_1\in \mathcal S_L \cap B(x_1,L) \\y_2\in \mathcal S_L \cap B(x_2,L) }} \pr{y_1\longleftrightarrow x_1, y_2\longleftrightarrow x_2 }\\ 
			& \lesssim L^4  + L^2 \sum_{u} \sum_{y\in \mathcal S_L\cap B(x_2,L)} \tau(x_1-u) G(u,y)\tau(x_2-u) + L^4 \sum_u  \tau(x_1-u)G(u,x_2). 
		\end{align*}
		Now we observe on one hand that, since $d\ge 7$, one has 	
		uniformly over $x_1,x_2\in \mathbb Z^d$,  
		\begin{equation}\label{convol.gG}
			\sum_u  \tau(x_1-u)G(u,x_2)\lesssim 1, 
		\end{equation}
		and on the other hand, since $x_1$ and $x_2$ are assumed to be at least at distance $L/4$ from $\mathcal S_L$, one has 
		\begin{align*} 
			\sum_{u} \sum_{y\in \mathcal S_L\cap B(x_2,L)} \tau(x_1-u) G(u,y)\tau(x_2-u) & \lesssim  \sum_{y\in \mathcal S_L\cap B(x_2,L)} \sum_{u\in B(y,L/8)} L^{2(2-d)}  G(u,y) \\
			& + \sum_{y\in \mathcal S_L\cap B(x_2,L)} \sum_{u\notin B(y,L/8)} \tau(x_1- u) G(u,y)\tau(x_2-u) \\
			\lesssim & L^{8-d}+ \sum_u \sum_{y\in B(x_2,L)\cap B(u,L/8)^c}\tau(x_1- u) G(u,y)\tau(x_2-u) \\ 
			\lesssim & L^{8-d} + L^d\cdot L^{4-d}\cdot \sum_u \tau(x_1-u)\tau(x_2-u)\\
			\lesssim & L^{8-d} + L^4 \cdot G(x_1,x_2).
		\end{align*}
		Taking the sum over all $x_1,x_2$ and using that $d\geq 7$ we finally deduce
		\begin{align*}
			& \sum_{x_1,x_2\in \text{Reg}'(B)}\sum_{\substack{y_1\in \mathcal S_L \cap B(x_1,L) \\y_2\in \mathcal S_L \cap B(x_2,L) }} \pr{y_1\longleftrightarrow x_1, y_2\longleftrightarrow x_2 } \lesssim L^4 \cdot M^2 + L^2 \cdot L^4 \cdot \sum_{x_1,x_2\in \text{Reg}'(B)} G(x_1,x_2).
			\end{align*}
		Since by definition $\text{Reg}'(B)$ satisfies the density condition required by Claim~\ref{cl:lowerboundoncapacity}, it follows that 
		\[
		\sum_{x_1,x_2\in \text{Reg}'(B)} G(x_1,x_2)\lesssim M,
		\]
		and hence using that $M\geq cL^2$, we conclude that 
		\begin{align*}
			& \sum_{x_1,x_2\in \text{Reg}'(B)}\sum_{\substack{y_1\in \mathcal S_L \cap B(x_1,L) \\y_2\in \mathcal S_L \cap B(x_2,L) }} \pr{y_1\longleftrightarrow x_1, y_2\longleftrightarrow x_2 } \lesssim L^4\cdot M^2.
			\end{align*}
		Putting all bounds together, we obtain for any $H$ with $|\text{Reg}'(B)|=M$, 
		\begin{align*}
			\E{Z_{r,L}^2\cdot \1(\ce{\ball{0}{r}}=H)} \lesssim M^2 \cdot L^4 \cdot \pr{\ce{\ball{0}{r}}=H}.
		\end{align*}
		Then taking the sum over all $H$ as above concludes the proof. \end{proof}

	\begin{proof}[\bf Proof of Theorem~\ref{thm:3}]
		Let $K$ be sufficiently large so that Lemma~\ref{lem:firstmomentofy1} and Proposition~\ref{thm:linegoodpoints} both hold. Fix also $L\ge 10K$ and $r\ge 2L$. Then by definition it is immediate to check that $A_{r,L}\geq Z_{r,L}$. Let~$\rho=\frac{c\cdot \epsilon(K)}{2}$, with $c$ as in Lemma~\ref{lem:firstmomentofy1} and $\epsilon(K)$ as in Proposition~\ref{thm:linegoodpoints}. We then obtain 
		\begin{align*}
			&	\pr{X_{r}\geq L^2 \text{ and } A_{r,L}\leq \rho L^4}\\ 
			&	 \leq \pr{X_{r}\geq L^2, \ X_{r}^{K-\line}\leq \epsilon(K) \cdot L^2} + \pr{X_{r}^{K-\line}\geq \epsilon(K)\cdot L^2, A_{r,L}\leq \rho \cdot L^4} \\
			&	\lesssim r^{d-1}\cdot \exp(-c\cdot (\log L)^4) +  \pr{X_{r}^{K-\line}\geq \epsilon(K)\cdot L^2, \ Z_{r,L}\leq \rho\cdot L^4},
		\end{align*}
		where for the last inequality we used Proposition~\ref{thm:linegoodpoints} and Claim~\ref{claim:Klinegood}.

		By the Payley-Zygmund inequality and using Lemmas~\ref{lem:firstmomentofy1} and~\ref{lem:secondmomentofy1}, for any $M\ge \epsilon(K)\cdot L^2$, 
		\begin{align*}
			\mathbb P\Big(Z_{r,L} \ge \rho\cdot  L^4\ \Big|\ X_{r}^{K-\line}=M\Big) 
			& \ge  \prcond{Z_{r,L}\geq \frac{1}{2}\econd{Z_{r,L}}{X_{r}^{K-\line}=M}}{X_{r}^{K-\line}=M}{} \\ 
			& \geq  \frac{1}{4}\cdot \frac{\left(\econd{Z_{r,L}}{X_{r}^{K-\line}=M}\right)^2}{\econd{Z_{r,L}^2}{X_{r}^{K-\line}=M}}\geq c_2=c_2(K), 
		\end{align*}
		for a positive constant $c_2$ depending on $K$. 
		Hence, summing over $M\ge \epsilon(K)\cdot L^2$, this gives 
		\begin{align*}
			\pr{X_{r}^{K-\line}\geq \epsilon(K)\cdot L^2, Z\leq \rho \cdot L^4}&\leq (1-c_2)\cdot  \pr{X_{r}^{K-\line}\geq \epsilon(K)\cdot L^2}
			\\ &\leq (1-c_2)\cdot \pr{X_r>0}.
		\end{align*}
		Using the obvious lower bound 
		\[
		\pr{0\longleftrightarrow \ball{0}{r}} \ge \frac{1}{r^{d-2}},
		\]
		and assuming now that $L\geq r^{\alpha}$, it follows that by taking $r$ sufficiently large 
		\[
		r^{d-1} \cdot \exp(-c\cdot (\log L)^4)\leq \frac{c_2}{2} \cdot \pr{0\longleftrightarrow \ball{0}{r}},
		\]
		and hence 
		\[
		\pr{X_{r}\geq L^2 \text{ and } A_{r,L}\leq cL^4} \leq \left(1-\frac{c_2}{2}\right)\cdot  \pr{0\longleftrightarrow \ball{0}{r}},
		\]
		which concludes the proof.
	\end{proof}

	\subsection{Hitting probabilities for the IIC}

		We start by proving the upper bound, which is by far the easiest part of the theorem, as it simply follows from an application of the BK inequality and the definition of the IIC measure. 
	
	\begin{proof}[\bf Proof of Theorem~\ref{thm.IIC}: the upper bound]
		First note that for any $w,z\in \mathbb Z^d$, and any finite $A\subset \mathbb Z^d$, using~\eqref{BKineq}, 
		\begin{align*}
			\mathbb P(z\longleftrightarrow A, z\longleftrightarrow w)&  \le \sum_{u\in \mathbb Z^d} \mathbb P(\{u\longleftrightarrow A\}\circ \{u\longleftrightarrow z\}\circ \{u\longleftrightarrow w\}) \\
			& \le \sum_{u\in \mathbb Z^d} \tau(w-u)\cdot \tau(z-u) \cdot \mathbb P(u\longleftrightarrow A), 
		\end{align*}
		from which we infer, using~\eqref{twopoint}, 
		$$\limsup_{\|w\|\to \infty} \frac{\mathbb P(z\longleftrightarrow A, z\longleftrightarrow w) }{\tau(w)} \le \sum_{u\in \mathbb Z^d} \mathbb P(u\longleftrightarrow A) \cdot \tau(z-u). $$ 
		Now it is not difficult to see, using the definition of the IIC measure in~\cite{CCHS25} and a localization procedure, similar as what we did in the proof of Theorem~\ref{thm.defpcap}, that we also have for any fixed $z$, 
		$$\mathbb P(\mathcal C_\infty(z) \cap A\neq \emptyset) = \lim_{\|w\|\to \infty} \frac{\mathbb P(z\longleftrightarrow A, z\longleftrightarrow w) }{\tau(w-z)}=\lim_{\|w\|\to \infty} \frac{\mathbb P(z\longleftrightarrow A, z\longleftrightarrow w) }{\tau(w)}. $$ 
		Moreover, Theorem~\ref{thm.defpcap} and~\eqref{twopoint} imply that 
		$$\limsup_{\|z\|\to \infty} \frac{\sum_{u\in \mathbb Z^d} \mathbb P(u\longleftrightarrow A) \cdot \tau(z-u)}{\|z\|^{d-4}}\le \textrm{pCap}(A) \cdot\limsup_{\|z\|\to \infty} \frac{\sum_{u\in \mathbb Z^d} \tau(u) \tau(z-u)}{\|z\|^{d-4}} \lesssim \textrm{pCap}(A). $$
		Altogether, this concludes the proof of the upper bound in Theorem~\ref{thm.IIC}. 
	\end{proof}

	We are now ready to finish the proof of Theorem~\ref{thm.IIC}. Given three sets $A_1,A_2,A_3$, we will write,  $A_1\stackrel{\text{off } A_2}{\longleftrightarrow} A_3$, for the event that the two sets $A_1$ and $A_3$ are connected within $\Z^d\setminus A_2$ (except possibly for the endpoints of the path connecting $A_1$ and $A_3$ which is allowed to belong to $A_2$).

	\begin{proof}[\bf Proof of Theorem~\ref{thm.IIC}: the lower bound]	As already mentioned, using the definition of the IIC and translation invariance we have 
		\begin{align*}
			\pr{\cC_\infty(z)\cap A\neq \emptyset} & = \lim_{\|w\|\to\infty} \prcond{z\longleftrightarrow A}{z\longleftrightarrow w}{} = \lim_{\|w\|\to\infty} \frac{\pr{z\longleftrightarrow A, z \longleftrightarrow w}}{\pr{w\longleftrightarrow z}} \\
			& =\lim_{\|w\|\to\infty} \frac{\pr{0\longleftrightarrow z-A, 0 \longleftrightarrow w}}{\tau(w)}.
		\end{align*}
		Let $z$ be such that $\|z\|>4\cdot \max_{a\in A} \|a\|$ and fixed for now. The goal is to show that taking $\|w\|$ large  we get 
		\begin{align}\label{eq:goalforlowerbound}
			\pr{0\longleftrightarrow z-A, 0 \longleftrightarrow w} \gtrsim  \tau(z) \cdot \pcap{A} \cdot \|z\|^2 \cdot \tau(w) \asymp \|z\|^{4-d} \cdot \pcap{A} \cdot \tau(w),
		\end{align}
		where the last equivalence follows from~\eqref{twopoint}. 
		For a possible realisation $H$ of the  extended cluster of the origin, we let $B= H\cap B(0,r)$ (which on the event $\ce{\ball{0}{r}}=H$, is just equal to $\mathcal C(0;B(0,r))$), and we say that $H$ is good, if $|\text{Reg}'(B)|\geq \frac{\varepsilon(K)}{K}\cdot r^2$, with $\varepsilon(K)$ as in Claim~\ref{claim:Klinegood} (note that $\text{Reg}'(B)$ is entirely determined by $H$). 
		To prove~\eqref{eq:goalforlowerbound}, let $r=\|z\|/2$ and observe that 
		\begin{align}\label{eq:decompose}
			\pr{0\longleftrightarrow z-A, 0 \longleftrightarrow w} \geq \sum_{H \text{ good}} \pr{0\longleftrightarrow z-A, 0\longleftrightarrow w, \ce{\ball{0}{r}} = H}.
		\end{align}
		Furthermore, for any fixed $H$, one has
		\begin{align*}
			&	\pr{0\longleftrightarrow z-A,\  0\longleftrightarrow w, \ \ce{\ball{0}{r}}=H} \\
			& \geq  \pr{z-A\stackrel{\text{off } H}{\longleftrightarrow} X_r(B)\cup L(B),\  w\stackrel{\text{off }  H}{\longleftrightarrow} \text{Reg}'(B),\  \ce{\ball{0}{r}}=H} \\
			& =\pr{z-A\stackrel{\text{off } H}{\longleftrightarrow} X_r(B)\cup L(B),\  w\stackrel{\text{off }  H}{\longleftrightarrow} \text{Reg}'(B)}\cdot \pr{\ce{\ball{0}{r}}=H},
		\end{align*}	
		where for the last equality we used the independence between the two events as they depend on disjoint sets of edges. Since the events $\{z-A \stackrel{\text{off } H}{\longleftrightarrow} X_r(B)\cup L(B)\}$ and $\{w\stackrel{\text{off } H}{\longleftrightarrow} \text{Reg}'(B)\}$ are increasing, we can apply the FKG inequality~\eqref{FKGineq} and deduce
		\begin{align}\label{eq:zandb}
			\begin{split}
				&\pr{0\longleftrightarrow z-A, 0\longleftrightarrow w, \ \ce{\ball{0}{r}}=H} \\
				&\geq \pr{z-A\stackrel{\text{off } H}{\longleftrightarrow} X_r(B)\cup L(B)}\cdot \pr{ w\stackrel{\text{off } H}{\longleftrightarrow} \text{Reg}'(B)}\cdot \pr{\ce{\ball{0}{r}}=H}\\&=\pr{ w\stackrel{\text{off } H}{\longleftrightarrow} \text{Reg}'(B)} \cdot 
				\pr{z-A\stackrel{\text{off } H}{\longleftrightarrow}X_r(B)\cup L(B),\ \ce{\ball{0}{r}}=H}\\
				&=\pr{ w\stackrel{\text{off } H}{\longleftrightarrow} R'(B)} \cdot 
				\pr{0\longleftrightarrow z-A,\ \ce{\ball{0}{r}}=H},
			\end{split}
		\end{align}
		where for the last equality we used that the only way for $z-A$ to be connected to $0$ is by having a connection to $X_r(B)\cup L(B)$ off $H$. 
		We next prove that 
		\begin{align}\label{eq:goalforw}
			\pr{ w\stackrel{\text{off } H}{\longleftrightarrow} R'(B)} \geq c(K) \cdot \pr{ w\stackrel{}{\longleftrightarrow} R'(B)},
		\end{align}	
		where $c(K)$ is a constant depending on $K$. 
		To this end assign an ordering to the points of $\text{Reg}'(B)$, and on the event $\{w\stackrel{\text{off } H}{\longleftrightarrow} \text{Reg}'(B)\}$, let $Y$ be the first point in the ordering such that $w$ is connected to the boundary of the ball of radius $K$ around this point off $H$. For $y\in \text{Reg}'(B)$, let $L_y$ be the line segment of length $K$ from $y$ to $\partial B(0,r)$ (which by definition is part of $H$), and let $$H^* = H\setminus (\cup_{y\in \text{Reg}'(B)}L_y).$$  
		We claim that
		\begin{equation}\label{eq:firstballhit}
			\pr{w\stackrel{\text{off } H}{\longleftrightarrow} \text{Reg}'(B) } \geq  \sum_{y\in \text{Reg}'(B)} \pr{Y=y, w\stackrel{\text{off } H^*}{\longleftrightarrow} \ball{y}{K},  \text{ all edges in } \ball{y}{K} \text{ are open}}.
		\end{equation}
		Indeed, notice that on the event in the probability above, there exists an open path connecting $w$ to $B(y,K)$ off $H$, since the lines $L_y$ are part of the balls $B(y,K)$. Let~$z$ be the vertex on $\partial B(y,K)$ that $w$ is connected to. Since all edges of $B(y,K)$ are open, this implies that there is an open path inside $B(y,K)$ connecting $z$ to $y$ without intersecting $L_y$, and hence this means that $w\stackrel{\text{off } H}{\longleftrightarrow} \text{Reg}'(B)$. 
		
		Since the two events appearing in the probability on the right hand side of~\eqref{eq:firstballhit} are independent, they factorise, and hence we get with $c(K)=p_c^{d(2K+1)^{d}}$,  
		\begin{align}\label{eq:hitoffb}
			\nonumber			\pr{w\stackrel{\text{off }H}{\longleftrightarrow} \text{Reg}'(B)}&\geq  c(K)\cdot \sum_{y\in \text{Reg}'(B)} \pr{Y=y, w\stackrel{\text{off } H^*}{\longleftrightarrow}  \ball{y}{K}}\\
			& = c(K) \cdot \mathbb P\Big(w\stackrel{\text{off }H^*}{\longleftrightarrow} \bigcup_{y\in \text{Reg}'(B)}B(y,K)\Big)
			\geq c(K)\cdot \pr{w\stackrel{\text{off }H^*}{\longleftrightarrow} \text{Reg}'(B)},
		\end{align}
		where the last inequality follows from the fact that to reach $\text{Reg}'(B)$, the cluster of $w$ must necessarily first reach $\cup_{y\in \text{Reg}'(B)}\partial B(y,K)$.
		It remains to bound the last probability above. First we write 
		\begin{equation}\label{eq:hitofftildb}
			\pr{w\stackrel{\text{off }H^*}{\longleftrightarrow} \text{Reg}'(B)} = \pr{w\longleftrightarrow \text{Reg}'(B)} -  \pr{w\stackrel{\text{via }H^*}{\longleftrightarrow} \text{Reg}'(B)}, 
		\end{equation}
		where ``via $H^*$" means using a path of open edges intersecting $H^*$.  
		Next by the BK inequality~\eqref{BKineq} and by taking $w$ sufficiently large we get,
		\begin{align}\label{eq:boundviabtilde}
			\nonumber		\pr{ w\stackrel{\text{via } H^*}{\longleftrightarrow} \text{Reg}'(B)}& \leq \sum_{b\in H^*} \pr{w\longleftrightarrow b} \cdot \pr{b\longleftrightarrow \text{Reg}'(B)} \lesssim \tau(w) \cdot \sum_{b\in H^*} \pr{b\longleftrightarrow \text{Reg}'(B)} \\ 
			&	\lesssim \tau(w) \cdot \sum_{x'\in \text{Reg}'(B)} \sum_{b\in H^*} \tau(x'-b), 
		\end{align}
		where the last inequality follows by a union bound. 
		We set for $t>0$
		\[
		B_t(x') = H^*\cap \ball{x'}{2^t}\setminus \ball{x'}{2^{t-1}}.
		\]
		Then the above sum can be split as follows, using~\eqref{twopoint}, 
		\begin{align*}
			\sum_{x'\in \text{Reg}'(B)} \sum_{b\in H^*} \tau(x'-b) \asymp \sum_{x'\in \text{Reg}'(B)}\sum_{t=0}^{\log_2 (2r)} \frac{|B_t(x')|}{2^{t(d-2)}}.
		\end{align*}
		Recalling that $H^*$ is at distance $K$ from $x'$, we see that 
		$B_t(x')=\emptyset$, for $t< \log_2(K)$.
		For $t\geq \log_2(K)$ writing $x\in \Reg{B}$ for the point on $\partial \ball{0}{r}$
		corresponding to $x'$ and using the density condition for $x$, i.e.\ that the event $\cT_s(x)$ holds, since $x\in \Reg{B}$, we obtain (taking also into account points in $L(B)$ which may be part of $H^*$), 
		\[
		|B_t(x')| \leq  |\ball{x}{2^{t+1}}\cap H^* | \leq (2^{t+1})^4(\log(2^{t+1}))^7+K(2^{t+1})^2 (\log(2^{t+1}))^7 \asymp t^7\cdot (2^{4t} + K \cdot 2^{2t}), 
		\]
		Inserting this bound in the previous sum, this yields 
		\begin{align*}
			\sum_{x'\in \text{Reg}'(B)} \sum_{b\in H^*} \tau(x'-b) & \asymp \sum_{x'\in \text{Reg}'(B)}\sum_{t=\log_2(K)}^{\log_2(2r)} \frac{|B_t(x')|}{2^{t(d-2)}}	\lesssim 
			\sum_{x'\in \text{Reg}'(B)}\sum_{t=\log_2(K)}^{\infty} \frac{t^7\cdot (2^{4t} + K \cdot 2^{2t})}{2^{t(d-2)}}\\ & \lesssim \frac{1}{\sqrt K} |\text{Reg}'(B)|, 
		\end{align*}
		where for the last inequality we use that $d>6$. 
		Plugging this bound into~\eqref{eq:boundviabtilde}
		we deduce 
		\begin{align}\label{eq:hitviatildbsmall}
			\pr{w\stackrel{\text{via }H^*}{\longleftrightarrow} \text{Reg}'(B)}\lesssim \frac{1}{\sqrt K} \cdot \tau(w) \cdot |\text{Reg}'(B)|.
		\end{align}
		On the other hand, for $\|w\|$ sufficiently large we get from Lemma~\ref{lem:hitdiameterdistance}
		\begin{align*}
			\pr{ w{\longleftrightarrow} \text{Reg}'(B)}\gtrsim \tau(w) \cdot \cpc{d-4}{\text{Reg}'(B)}\gtrsim \tau(w) \cdot |\text{Reg}'(B)|,
		\end{align*}
		where for the second inequality we used Claim~\ref{cl:lowerboundoncapacity}, since the set $\text{Reg}'(B)$ satisfies the density condition of the statement of the claim by definition.
		Plugging the above together with~\eqref{eq:hitviatildbsmall} into~\eqref{eq:hitofftildb} and taking $K$ sufficiently large gives
		\begin{equation}\label{eq:goalforw2}
			\pr{w\stackrel{\text{off }H^*}{\longleftrightarrow} \text{Reg}'(B)} \gtrsim \tau(w) \cdot |\text{Reg}'(B)|.
		\end{equation}
		Finally, substituting this lower bound into~\eqref{eq:hitoffb}, and using that by a union bound we also have $\mathbb P(w\longleftrightarrow \text{Reg}'(B))\lesssim \tau(w) \cdot |\text{Reg}'(B)|$, this concludes the proof of~\eqref{eq:goalforw}. Plugging~\eqref{eq:goalforw} into~\eqref{eq:zandb} and then into~\eqref{eq:decompose} and using the definition of $H$ good gives 
		\begin{align*}
			\pr{0\longleftrightarrow z-A,0\longleftrightarrow w} \gtrsim  \tau(w) \cdot r^2 \cdot \pr{0\longleftrightarrow z-A,\  \ce{\ball{0}{r}} \text{ is good}}.
		\end{align*}
		It remains to bound this last probability above. Recalling the definition of a good extended cluster, i.e.\ a cluster such that $X_r^{K-\line}\geq \frac{\epsilon(K)}{K} \cdot r^2$, with $\epsilon(K)$ as in Claim~\ref{claim:Klinegood}, for $\delta>0$ to be determined and $K>2/\delta$, we have 
		\begin{align*}
			&\pr{0\longleftrightarrow z-A, \ce{\ball{0}{r}} \text{ is good}} = \pr{0\longleftrightarrow z-A, X_r^{K-\line}\geq \frac{\epsilon(K)}{K}\cdot r^2} \\&= \pr{0\longleftrightarrow z-A} -\pr{0\longleftrightarrow z-A, X_r^{K-\line}< \frac{\epsilon(K)}{K}\cdot r^2}\\ &\gtrsim   \tau(z)\cdot \pcap{A} -  \pr{0\longleftrightarrow z-A, 0<X_r\leq  \delta \cdot r^2} \\&\quad \quad \quad \quad \quad \quad \quad -  \pr{0\longleftrightarrow z-A, X_r>  \delta \cdot r^2, X_r^{K-\line} < \frac{\epsilon(K)}{2} \cdot X_r}, 
		\end{align*}
		where the first inequality is obtained by using Theorem~\ref{thm.defpcap} and taking $r$ large enough. 
		Concerning the first probability on the right-hand side above, note that conditionally on $\mathcal C(0;B(0,r))=B$, for $z-A$ to be connected to $0$, it must be connected to one of the pioneer points (recall that the set of pioneer points is denoted by $X_r(B)$) off $B$. Hence a union bound over these points gives, using again Theorem~\ref{thm.defpcap}, and taking $r$ large enough, 
		\begin{align*}
			\pr{0\longleftrightarrow z-A, 0<X_r\leq  \delta \cdot r^2} & = \sum_{B: 0<X_r(B)\le \delta r^2} \pr{0\longleftrightarrow z-A, \mathcal C(0;B(0,r))= B} \\
			& = \sum_{B: 0<X_r(B)\le \delta r^2} \pr{z-A\stackrel{\text{off } B}{\longleftrightarrow} X_r(B)} \cdot \mathbb P( \mathcal C(0;B(0,r))= B) \\
			& \le \sum_{B: 0<X_r(B)\le \delta r^2} \pr{z-A \longleftrightarrow X_r(B)} \cdot \mathbb P( \mathcal C(0;B(0,r))= B) \\
			& \lesssim \delta \cdot r^2 \cdot \tau(z) \cdot \textrm{pCap}(A)\cdot \mathbb P(X_r\ge 1).
		\end{align*}
		Using next the one arm exponent~\eqref{onearm} we obtain
		\begin{align*}
			\pr{0\longleftrightarrow z-A, 0<X_r\leq  \delta \cdot r^2} \lesssim  \delta \cdot \tau(z) \cdot \pcap{A}.
		\end{align*}
		Similarly, by conditioning on the configuration in $\ball{0}{r}$ as before, we get 
		\begin{align*}
			& \pr{0\longleftrightarrow z-A, X_r>  \delta \cdot r^2, X_r^{K-\line} < \frac{\epsilon(K)}{2} \cdot X_r} \\ &\lesssim  \E{X_r\cdot \1(X_r>  \delta \cdot r^2, X_r^{K-\line} < \frac{\epsilon(K)}{2} \cdot X_r)\cdot \tau(z) \cdot \pcap{A}} \\
			&\lesssim r^{d-1}  \cdot \tau(z) \cdot \pcap{A} \cdot \pr{X_r>  \delta \cdot r^2, X_r^{K-\line} < \frac{\epsilon(K)}{2} \cdot X_r}\\
			&\lesssim r^{2d}\cdot \exp(-c(\log r)^4) \cdot \tau(z) \cdot \pcap{A} \lesssim \exp(-c(\log r)^4/2) \cdot \tau(z) \cdot \pcap{A},
		\end{align*}
		where for the second inequality we used the crude bound $X_r\lesssim r^{d-1}$, and for the last one we used Proposition~\ref{thm:linegoodpoints} and Claim~\ref{claim:Klinegood}. Taking $\delta$ sufficiently small, $r$ sufficiently large and using the two bounds above we finally conclude that
		\begin{align*}
			\pr{0\longleftrightarrow z-A, 0\longleftrightarrow w} \gtrsim  \tau(w) \cdot r^2 \cdot \tau(z)\cdot \pcap{A},
		\end{align*}
		and this finishes the proof.
	\end{proof}

	
	\section{p-Capacity of a ball}\label{sec:pCapball}
	\label{sec:pcapball}
	In this section we assume that $d\ge 8$ and we prove Theorem~\ref{thm.pcap.ball}. We first introduce some notation and state a result whose proof we defer to Section~\ref{sec:pointsinannulus}.

	Let $z\in \Z^d$. 
	For $r\ge 1$ we write $\partial B(0,r)^c$ for the (inner) boundary of $B(0,r)^c$, i.e.~the set of points in $B(0,r)^c$ which have at least one neighbor in $B(0,r)$, and define for $L<r<\|z\|$, 
	\[
	X_r = \left|\{x\in \partial B(0,r)^c: z\stackrel{\text{off }B(0,r)}{\longleftrightarrow x}\}\right| \quad \text{ and }\quad  A_{r,L} = \left|\{x\in \ball{0}{r}\setminus \ball{0}{r-L}: z\longleftrightarrow  x\}\right|.
	\]
	The next proposition is the core of the whole proof of Theorem~\ref{thm.pcap.ball}. The proof follows very closely the proof of Lemma~2.3 in~\cite{KN11}.
	\begin{proposition}\label{thm:2}
		There exists a constant $c>0$ so that for every constant $c'>0$, all $r$ sufficiently large, all $L\geq c'r$, we have for all $z$ with $\|z\|> r$, 
		\[
		\pr{X_{r}\geq L^2 \  \text{ and } \ A_{r,L}\leq cL^4} \leq (1-c)\cdot \pr{\cC(z)\cap B(0,r)\neq \emptyset}.
		\]
	\end{proposition}
	
We defer the proof of the above proposition to Section~\ref{sec:pointsinannulus}.

\subsection{Proof of Theorem~\ref{thm.pcap.ball}}

In this section we give the proof of Theorem~\ref{thm.pcap.ball} after stating and proving two preliminary results that could be of independent interest. We note that alternatively it is also possible to use an induction argument similar to the proof of~\cite[Theorem~1]{KN11}, but we do not include it here.

	\begin{proposition}\label{pro:green}
		Let $\lambda>0$. For each $R>0$ let $\cG_{\lambda,R}$ be the random set obtained by keeping each point of $\ball{0}{R}$ independently with probability $1-e^{-\lambda}$. We then have 
		\[
		\E{\pcap{\cG_{\lambda,R}}} = \pcap{\ball{0}{R}} \cdot \left(1 - \lim_{\|z\|\to\infty}\econd{e^{-\lambda |\cC(z)\cap \ball{0}{R}|}}{z\longleftrightarrow \ball{0}{R}}   \right).
		\]
	\end{proposition}
	
	\begin{proof}[\bf Proof]
		
		Using that each point of $\ball{0}{R}$ is included in $\cG_{\lambda,R}$ independently we obtain by conditioning on $\cC(z)$
		\begin{align}\label{eq:noninters}
			\begin{split}
				\pr{\cC(z)\cap \cG_{\lambda,R}\neq \emptyset}= 1 - \E{\prcond{\cC(z)\cap \cG_{\lambda,R}= \emptyset}{\cC(z)}{}} =1 - \E{e^{-\lambda  |\cC(z)\cap \ball{0}{R}|}}\\
				=\pr{z\longleftrightarrow \ball{0}{R}}\left( 1 - \econd{e^{-\lambda  |\cC(z)\cap \ball{0}{R}|}}{z\longleftrightarrow \ball{0}{R}}\right).
			\end{split}
		\end{align}
		By conditioning now with respect to $\cG_{\lambda,R}$ we get 
		\begin{align*}
			\lim_{\|z\|\to\infty} \frac{1}{\tau(z)} \cdot \pr{\cC(z)\cap \cG_{\lambda,R}\neq \emptyset} =	\lim_{\|z\|\to\infty} \frac{1}{\tau(z)} \cdot \E{\prcond{\cC(z)\cap \cG_{\lambda,R}\neq \emptyset}{\cG_{\lambda,R}}{}} = \E{\pcap{\cG_{\lambda,R}}}.
		\end{align*}
		Using the above and plugging it into~\eqref{eq:noninters} after we first divide both sides by $\tau(z)$ gives
		\begin{align*}
			\E{\pcap{\cG_{\lambda,R}}} = \pcap{\ball{0}{R}}\cdot \left(1 - \lim_{\|z\|\to\infty}\econd{e^{-\lambda  |\cC(z)\cap \ball{0}{R}|}}{z\longleftrightarrow \ball{0}{R}}   \right)
		\end{align*}
		and this concludes the proof.
	\end{proof}

	\begin{lemma}\label{lem:rsquaredwhenithits}
		There exists $\delta>0$ so that for all $R$ and for all $\|z\|$ sufficiently large we have 
		\[
		\prcond{X_R>\delta R^2}{X_R> 0}{} \geq \delta.
		\]
	\end{lemma}
	
	\begin{proof}[\bf Proof]
		It is sufficient to prove that there exists a universal positive constant $C$ so that for all $\|z\|$ large enough and all $\delta>0$ we have 
		\begin{align}\label{eq:twoproofs}
			&	\pr{X_{2R}>0}\leq C \cdot \pr{X_R>0} \quad \text{ and } \\ &\label{eq:secondproof}\pr{X_R>0,X_{2R}<\delta R^2}\leq C\cdot \delta \cdot \pr{0<X_{2R}<\delta R^2}.
		\end{align}
		Indeed, once this is established, then since $X_R>0$ implies that $X_{2R}>0$, we would get 
		\begin{align*}
			\frac{1}{C}\cdot 	\pr{X_{2R}>0}\leq \pr{X_R>0} &\leq C\cdot \delta \cdot \pr{0<X_{2R}<\delta R^2} + \pr{X_{R}>0, X_{2R}>\delta R^2} \\
			&	\leq C\cdot \delta\cdot  \pr{X_{2R}>0} + \pr{X_{2R}>\delta R^2, X_{2R}>0}.
		\end{align*}
		By now choosing $\delta=1/(2C^2)$ and rearranging the above gives the desired result. We now first prove~\eqref{eq:twoproofs}. To show this, we use the elementary fact that a ball of radius $2R$ is contained in a finite union of balls of radius $R$. More precisely, there exists $K>0$ so that for all $R>0$ there exist $(a_i)_{i\leq K}$ with $|a_i|\leq 2R$ satisfying 
		\[
		\ball{0}{2R}\subseteq \bigcup_{i=1}^{K} \ball{a_i}{R}.
		\]
		Therefore, by a union bound we get 
		\begin{align*}
			\pr{z\longleftrightarrow \ball{0}{2R}} \leq \sum_{i=1}^{K}\pr{z\longleftrightarrow \ball{a_i}{R}}.
		\end{align*}
		By taking $\|z\|$ sufficiently large we see by Theorem~\ref{thm.defpcap} that for positive constants $C_1$ and $C_2$ we have 
		\begin{align*}
			C_1\cdot \tau(z) \cdot \pcap{\ball{a_i}{R}}\leq \pr{z\longleftrightarrow \ball{a_i}{R}} \leq C_2 \cdot \tau(z) \cdot \pcap{\ball{a_i}{R}}.
		\end{align*}
		But by translation invariance of $\pcap{\cdot}$ we have that $\pcap{\ball{a_i}{R}} = \pcap{\ball{0}{R}}$ for all $i$, and hence using once more Theorem~\ref{thm.defpcap} we deduce 
		\begin{align*}
			\tau(z) \cdot \pcap{\ball{0}{R}} \leq \frac{1}{C_1} \cdot \pr{z\longleftrightarrow \ball{0}{R}}.
		\end{align*}
		Combining all of the above finally shows~\eqref{eq:twoproofs}.  It remains to prove~\eqref{eq:secondproof}. Let $C$ be a large positive constant. Then 
		\begin{align*}
			&\pr{X_{2R}<\delta R^2, X_R>0, |\cC(z;\ball{0}{2R}^c)|\leq C} 
			\\&= \sum_{\substack{|A|\leq C \\0< |A\cap \partial\ball{0}{2R}|<\delta R^2 }} \prcond{X_R>0}{\cC(z; \ball{0}{2R}^c)=A}{} \times \pr{\cC(z;\ball{0}{2R}^c)=A}.
		\end{align*}
		We now have 
		\begin{align*}
			&\prcond{X_R>0}{\cC(z; \ball{0}{2R}^c)=A}{} \\&=\prcond{\exists \ x\in A\cap \partial \ball{0}{2R}: x\stackrel{\text{off } A}{\longleftrightarrow} \ball{0}{R}}{\cC(z;\ball{0}{2R}^c)=A}{} \\
			&= \pr{\exists \ x\in A\cap \partial \ball{0}{2R}: x\stackrel{\text{off } A}{\longleftrightarrow} \ball{0}{R}}
			\lesssim|A\cap \partial \ball{0}{2R}| \cdot \frac{1}{R^2},
		\end{align*}
		where for the second equality we used the independence of the two events as they depend on disjoint sets of edges and for the last inequality we used a union bound together with~\eqref{onearm}. Therefore, putting everything together gives 
		\begin{align*}
			\pr{X_{2R}<\delta R^2, X_R>0, |\cC(z;\ball{0}{2R}^c)|\leq C} \lesssim \delta R^2 \cdot \frac{1}{R^2} \cdot \pr{0<X_{2R}<\delta R^2, |\cC(z;\ball{0}{2R}^c)|\leq C} \\
			= \delta \cdot \pr{0<X_{2R}<\delta R^2, |\cC(z;\ball{0}{2R}^c)|\leq C}.
		\end{align*}
		Taking the limit as $C\to\infty$ finally finishes the proof.
	\end{proof}

	\begin{proof}[\bf Proof of Theorem~\ref{thm.pcap.ball}]
		Let $\lambda=1/R^4$ and consider $\cG_{\lambda,R}$ as in the statement of Proposition~\ref{pro:green}. Then we have 
		\begin{align*}
			\E{\pcap{\cG_{\lambda,R}}} \leq \E{|\cG_{\lambda,R}|} = (1-e^{-\lambda})\cdot |\ball{0}{R}|\asymp \lambda \cdot R^d = R^{d-4}.
		\end{align*}
		Applying Proposition~\ref{pro:green} and the above we then obtain for some positive constant $\kappa$
		\begin{align*}
			\kappa \cdot R^{d-4} \geq \pcap{\ball{0}{R}} \cdot \left(1-\lim_{\|z\|\to\infty} \econd{e^{-|\cC(z)\cap \ball{0}{R}|/R^4}}{z\longleftrightarrow \ball{0}{R}}     \right).
		\end{align*}
		To conclude the proof it suffices to show that there exists a constant $c>0$ so that 
		\[
		\prcond{|\cC(z)\cap\ball{0}{R} |>c R^4}{z\longleftrightarrow \ball{0}{R}}{}\geq c.
		\]
		Letting $c$ be as in Proposition~\ref{thm:2} we obtain
		\begin{align*}
			&	\prcond{|\cC(z)\cap\ball{0}{R} |>c R^4}{z\longleftrightarrow \ball{0}{R}}{}\geq \prcond{A_{R,R/2}>cR^4}{X_R>0}{} \\ 
			&\geq \prcond{A_{R,R/2}>cR^4,X_{R}>\delta R^2}{X_R>0}{} = \frac{\pr{A_{R,R/2}>cR^4, X_R>\delta R^2}}{\pr{X_R>0}}\\
			&\geq  \frac{\pr{X_R>\delta R^2}- c\cdot \pr{X_R>0}}{\pr{X_R>0}} = (1-c)\cdot \prcond{X_R>\delta R^2}{X_R>0}{} \geq \delta \cdot (1-c),
		\end{align*}
		where for the second inequality we used Proposition~\ref{thm:2} and for the last one we used Lemma~\ref{lem:rsquaredwhenithits}. This now concludes the proof.
	\end{proof}

	\subsection{Points in an annulus -- Proof of Proposition~\ref{thm:2}}\label{sec:pointsinannulus}
	
	The proof of this proposition is based on two intermediate results, that we state and prove now. We start with some definitions regarding regularity of the cluster similarly to Section~\ref{sec.regpoints}. 
	By replacing $\ball{0}{r}$ with $\ball{0}{r}^c$ in the definitions at the beginning of Section~\ref{sec.regpoints} and before the proof of Claim~\ref{claim:Klinegood} we get the analogous definitions of global and local density conditions and also for regular and line-good points. Note that in the surface deviations of the cluster we need to replace $s^2(\log s)^7$ and $s^2(\log s)^4$ by $s^3(\log s)^7$ and $s^3(\log s)^4$ respectively. Indeed, we bound the restricted Green function of the complement of the box by the whole-space Green function.  As a consequence we have to assume $d\ge 8$ here, but we note that the forthcoming~\cite{ASS26+} provides another proof which covers also the case of dimension $7$.

	Now fix a constant $K\ge 1$, and to each $x\in \partial \ball{0}{r}$ we associate a connected path $L_x$ of length at most $dK$,  starting from $x$ and contained in $B(x,K)$, such that its endpoint, say $x'$, belongs to $\partial B(0,r-K)$.  
	Note that when $x$ is at distance at least $K$ from the boundaries of a face of the cube $B(0,r)$, one can just take for $L_x$ a line segment of length $K$ orthogonal to the face of the cube. Moreover, similarly as before, given $B$ a connected subgraph of $\ball{0}{r}^c$, we consider a maximal subset of the set of  $K$-regular points (now corresponding to $\ball{0}{r}^c$) of $B$ that are within distance at least $4K$ from each other and we call it $\Reg{B}$ (note that by construction, for any~$x,y\in \Reg{B}$ the corresponding paths $L_x$ and $L_y$ are at distance at least $2K$ from each other).  
	For each $x\in \Reg{B}$, if all the edges of $L_x$ are open, then we call $x$ a $K$-line-good point of $B$, and let $\text{Reg}'(B)$ be the union of all points $x'$ such that $x$ is a $K$-line good point. 
	Furthermore, for each~$x\in \Reg{B}$ we  consider the maximal connected subgraph of $L_x$ containing $x$ formed by open edges and denote by $L(B)$ the union of all these subgraphs.

	Let now $r\ge 1$ and $z$ be such that $\|z\|>r$. As in Section~\ref{sec:hitting} 
	we write $\cez{\ball{0}{r}^c}$ for the ``extended'' cluster of $z$, that is 
	\[
	\cez{\ball{0}{r}^c} = \cC(z;\ball{0}{r}^c) \cup L(\cC(z;\ball{0}{r}^c)), 
	\]
	and let $X_r^{K-\line}$ be the number of $K$-line good points of $\cC(z;\ball{0}{r}^c)$. We also write for simplicity 
	$$\mathcal R' = \text{Reg}'(\cC(z;\ball{0}{r}^c)).$$

	Finally for $L\ge 10K$ we define $\mathcal S_L = B(0,r-L/2) \setminus B(0,r-L)$, and
	\[
	Z_{r,L} = \sum_{y\in \mathcal S_L} \1\big(y\stackrel{\text{off } \cez{B(0,r)^c}}{\longleftrightarrow }
	\mathcal R'\cap B(y,L)\big).
	\]

	The fact that in this definition we consider only connections between points $y\in \mathcal S_L$ to other points in the set $\mathcal R'$ at distance at most $L$ from each other, ensures that the constant $c$ in the next lemma is independent of $L$, see also the explanation just after~\eqref{eq:goalforw2}. 
		
	\begin{lemma}\label{lem:firstmomentofy}
		There exists $c>0$, such that for all $K$ sufficiently large, all $L\geq 10K$, $M\ge 1$, and $r\ge 2L$, we have 
		\[
		\E{Z_{r,L}\cdot \1(X_r^{K-\line}=M)} \ge c L^2 \cdot M\cdot 
		\pr{X_r^{K-\line}=M}.
		\]	
	\end{lemma}

	\begin{lemma}\label{lem:secondmomentofy}
		There exists $C>0$, such that for all $L\ge 10K$, $M\ge 1$, and $r\ge 2L$, we have 
		\begin{align*}
			\E{Z_{r,L}^2\cdot \1(X_r^{K-\line}=M)}\leq C\cdot  M^2\cdot L^4\cdot \pr{X_r^{K-\line}=M}.
		\end{align*}
	\end{lemma}
	
	The proofs of these two lemmas follow similarly to the proofs of  Lemmas~\ref{lem:firstmomentofy1} and~\ref{lem:secondmomentofy1}. We explain the required changes. First we need to replace $\ce{\ball{0}{r}}$ by $\cez{\ball{0}{r}^c}$. Because the $\sigma$-algebra generated by the extended cluster in $\ball{0}{r}^c$  is uncountably generated, we cannot sum over all possible realisations of $\cez{\ball{0}{r}^c}$ as we did in the previous section. To overcome this issue, we introduce the indicator $\1(|\cez{\ball{0}{r}^c}|\leq C)$ in both the first and the second moments above and then using monotone convergence as $C\to\infty$ we obtain the desired bounds.

	\begin{proposition}\label{thm:localgoodforball}
		For all $K$ sufficiently large, there exists a constant $\alpha(K)>0$, so that for all $  1\le M \le Kr^{d-1}$, and all $ r\leq \|z\|/2$, we have 
		\[
		\pr{X_r\ge  M, X_r^{K-\line}\leq \alpha(K) \cdot X_r}\leq \tau(z) \cdot r^{2d-2}\cdot \exp(-\alpha(K)\cdot (\log M)^4).
		\]
	\end{proposition}
\begin{proof}[\bf Proof]
		To prove the proposition, we first control the number of $K$-regular points and show a concentration estimate analogous to Proposition~\ref{thm:linegoodpoints} but with the extra $\tau(z)$ factor as in the statement above. As in Proposition~\ref{thm:linegoodpoints}, it is enough to control the number of locally regular points. 
		By definition we have 
		\[
		X^{K-\rm{irr}}_r = \sum_{s= K}^{2r} X_r^{s-\rm{loc-bad}},
		\]
		(indeed, note that one can stop the sum at $s=2r$, since if a point is $s$-locally bad, for some $s\ge 2r$, then by definition it is also $2r$-locally bad) and hence it suffices to bound 
		\[
		\pr{X_r\geq M \ \text{ and } \ X_r^{s-\rm{loc-bad}}>X_r/s^2}. 
		\]
		Let $R=r+ 4s^{d}$. We start with the case $R< \|z\|$.  We then have 
		\begin{align*}
			\pr{X_r\geq M \ \text{ and } \ X_r^{s-\rm{loc-bad}}>\frac{X_r}{s^2}} =\prcond{X_r\geq M \ \text{ and } \ X_r^{s-\rm{loc-bad}}>\frac{X_r}{s^2}}{X_R\geq 1}{}\pr{X_R\geq 1}.
		\end{align*}
		First we notice that by a union bound $\pr{X_R\geq 1}\lesssim \tau(z) R^{d-1}$. On the event $\{X_R\geq 1\}$ we let $W$ be the set of vertices on $\partial B(0,R)^c$ that connect to $z$ via open paths of edges lying entirely in $B(0,R)^c$. We now repeat the exploration procedure as in the proof of Proposition~\ref{thm:linegoodpoints} but with $0$ replaced with $W$ and exploring all the boxes of the cluster of $z$ intersecting $\partial B(0,r)$. As before, we consider partitions into boxes of side length $4s^{d}$. In particular, as for~\eqref{LDsloc} we get for some constant $\kappa >0$, 
		\[
		\prcond{X_r\geq M \ \text{ and } \ X_r^{s-\loc-\bad}\geq \frac{X_r}{s^2}}{X_R\geq 1}{}\lesssim \exp\left(- \kappa\cdot  \frac{M}{s^{2d^2+4}}\right).
		\]
		Taking the sum over all $s \in [K, M^{1/(4d^2+8)}]$ we get for some possibly smaller constant $\kappa>0$
		\[
		\sum_{s=K}^{M^{1/(4d^2+8)}} 	\prcond{X_r\geq M \ \text{ and } \ X_r^{s-\loc-\bad}\geq \frac{X_r}{s^2}}{X_R\geq 1}{}\lesssim \exp(-\kappa\cdot \sqrt{M}).
		\]
		For $s\ge M^{1/(4d^2+8)}$ (and still satisfying $R< \|z\|$) we  use a union bound and Claim~\ref{claim:Tsloc} which holds for the new definition of $\cT_s^{\rm{loc}}$ because of~\eqref{eq:badbound}. Then		
		\begin{align*}
			\prcond{X_r\geq M \text{ and } X^{s-\rm{loc-bad}}_r> X_r/s^2}{X_R\geq 1}{} \leq \pr{X^{s-\rm{loc-bad}}_r\ge 1} \\ \lesssim r^{d-1}\cdot \exp(-c(\log s)^4).
		\end{align*}
		Likewise, if $R>\|z\|$, then  note that it implies $s^d>\|z\|/8$ (recall that $r\le \|z\|/2$ by hypothesis) and hence by a union bound without conditioning on $X_R\geq 1$, we obtain
		\begin{align*}
			\pr{X_r\geq M \ \text{ and } \ X_r^{s-\rm{loc-bad}}>X_r/s^2} &  \lesssim r^{d-1} \cdot \exp(-c(\log s)^4)\\
			& \lesssim \tau(z) \cdot r^{d-1} \cdot \exp(-c'(\log M)^4). 
		\end{align*}
		Putting all these bounds together shows 
		\[
		\pr{X_r\geq M, X_r^{K-\reg}\leq \frac{X_r}{2}} \lesssim \tau(z) \cdot r^{2d-2}\cdot \exp(-c_1 (\log M)^4),
		\]
		for a positive constant $c_1$. To control the number of $K$-line-good points, on the event that there are enough $K$-regular points, we use that each $K$-regular point is independently a $K$-line-good point with probability at least $p_c^{dK}$. We thus get using binomial concentration again
		\begin{align*}
			\pr{X_r\geq  M, X_r^{K-\reg}\geq \frac{X_r}{2}, X_r^{K-\line} \leq X_r^{K-\reg}\cdot p_c^{dK}/K^{d+1}} \leq \tau(z) \cdot r^{d-1}\cdot \exp(-\alpha(K)\cdot M),
		\end{align*}
		with $\alpha(K)$ a positive constant depending on $K$. Note that here we used that $X_r\geq M$ implies in particular that $\cC(z)\cap \ball{0}{r}\neq \emptyset$ and this has probability bounded from above by $\tau(z)\cdot r^{d-1}$ by a union bound.  Putting everything together completes the proof.
	\end{proof}

	\begin{proof}[\bf Proof of Proposition~\ref{thm:2}]
		Let $K$ be sufficiently large, as in Lemma~\ref{lem:firstmomentofy} and Proposition~\ref{thm:localgoodforball}. Fix also $L\ge 10K$ and $r\ge 2L$. Then by definition it is immediate to check that $A_{r,L}\geq Z_{r,L}$. Let~$\rho=\frac{c\cdot \alpha(K)}{2}$, with $c$ as in Lemma~\ref{lem:firstmomentofy} and $\alpha(K)$ as in Proposition~\ref{thm:localgoodforball}. We then obtain 
		\begin{align*}
			&	\pr{X_{r}\geq L^2 \text{ and } A_{r,L}\leq \rho L^4}\\ 
			&	 \leq \pr{X_{r}\geq L^2, \ X_{r}^{K-\line}\leq \alpha(K) \cdot L^2} + \pr{X_{r}^{K-\line}\geq \alpha(K)\cdot L^2, A_{r,L}\leq \rho \cdot L^4} \\
			&	\lesssim \tau(z) \cdot r^{2d-2}\cdot \exp(-\alpha(K)\cdot (\log L)^4) +  \pr{X_{r}^{K-\line}\geq \alpha(K)\cdot L^2, \ Z_{r,L}\leq \rho\cdot L^4},
		\end{align*}
		where for the last inequality we used Proposition~\ref{thm:localgoodforball}. By the Payley-Zygmund inequality and using Lemmas~\ref{lem:firstmomentofy} and~\ref{lem:secondmomentofy}, for any $M\ge \alpha(K)\cdot L^2$, 
		\begin{align*}
			\mathbb P\Big(Z_{r,L} \ge \rho\cdot  L^4\ \Big|\ X_{r}^{K-\line}=M\Big) 
			& \ge  \prcond{Z_{r,L}\geq \frac{1}{2}\econd{Z_{r,L}}{X_{r}^{K-\line}=M}}{X_{r}^{K-\line}=M}{} \\ 
			& \geq  \frac{1}{4}\cdot \frac{\left(\econd{Z_{r,L}}{X_{r}^{K-\line}=M}\right)^2}{\econd{Z_{r,L}^2}{X_{r}^{K-\line}=M}}\geq c_2=c_2(K), 
		\end{align*}
		for a positive constant $c_2$ depending on $K$. 
		Hence, summing over $M\ge \alpha(K)\cdot L^2$, this gives 
		\begin{align*}
			\pr{X_{r}^{K-\line}\geq \alpha(K)\cdot L^2, Z\leq \rho \cdot L^4}&\leq (1-c_2)\cdot  \pr{X_{r}^{K-\line}\geq \alpha(K)\cdot L^2}
			\\ &\leq (1-c_2)\cdot \pr{z\longleftrightarrow \ball{0}{r}}.
		\end{align*}
		Using the obvious lower bound 
		\[
		\pr{z\longleftrightarrow \ball{0}{r}} \ge \tau(z),
		\]
		and assuming now that $L\geq c'r$, for some constant $c'>0$, it follows that by taking $r$ sufficiently large 
		\[
		\tau(z)\cdot r^{d-1} \cdot \exp(-c_1(\log L)^4)\leq \frac{c_2}{2} \cdot \pr{z\longleftrightarrow \ball{0}{r}},
		\]
		and hence 
		\[
		\pr{X_{r}\geq L^2 \text{ and } A_{r,L}\leq cL^4} \leq \left(1-\frac{c_2}{2}\right)\cdot  \pr{z\longleftrightarrow \ball{0}{r}},
		\]
		which concludes the proof.
	\end{proof}


\section{Hitting probability of a tile of an hyperplane and improved large deviations bound}\label{sec:tile}

In this section we improve our large deviations bounds in dimension $d\ge 8$. 
Recall that $\mathcal H$ denotes the half-space $\{z\in \mathbb Z^d : z_1\ge 0\}$, and given $x\in \mathcal H$ and $s>0$, we let $Q_s = B(0,s)\cap \partial \mathcal H$,
and $r_x = 1+ d(x,\partial \mathcal H)$, $D_x = 1+ d(x,Q_s)$.

\begin{theorem}\label{thm.improvedLD}
Assume $d\ge 8$. There exists a constant $C>0$, such that for any $s\ge 1$ and any $x\in \mathcal H$, one has 
\begin{equation}\label{prob.tile}
\mathbb P(x\stackrel{\mathcal H}{\longleftrightarrow} Q_s) \le Cr_x\cdot \frac{(s\wedge D_x)^{d-3}}{D_x^d}. 
\end{equation}
Consequently, there exists $c,C>0$, such that for any $s,t>0$, and any $x\in \mathcal H$, 
$$\mathbb P\big(|\mathcal C(x;\mathcal H)\cap Q_s|\ge t\big)\le Cr_x\cdot \frac{(s\wedge D_x)^{d-3}}{D_x^d}\cdot \exp(-ct/s^2). $$    
\end{theorem} 	
Note that a similar large deviation bound can be deduced for the probability to hit a tile of a box, as in Theorem~\ref{thm.LD}.  

\begin{proof}[\bf Proof]
The second assertion of the theorem follows from the first one by substituting the bound~\eqref{prob.tile} in~\eqref{eq:usefulfor8} in the proof of Theorem~\ref{thm.LD}. So we just need to prove~\eqref{prob.tile}.  
Assume first that $4s\le D_x$, and let $\Lambda_s = \partial B(0,2s)\cap \mathcal H$. Let 
$$X_s^x = \sum_{u\in \Lambda_s} \1\big(x \stackrel{\mathcal H \cap B(0,2s)^c}{\longleftrightarrow} u\big), $$ 
be the number of pioneers (seen from $x$) on $\Lambda_s$. If $x$ is connected to $Q_s$ by an open path, there exist two disjoint open paths, one from 
$x$ to $\Lambda_s$ contained in $\mathcal{H}\setminus B(0,2s)^c$ and another one from $\Lambda_s$ to $Q_s$. Therefore, using the BK inequality~\eqref{BKineq}, and the one-arm estimate~\eqref{onearm}, we get 
$$\mathbb P(x\stackrel{\mathcal H}{\longleftrightarrow} Q_s) \le \frac{C}{s^2} \cdot \mathbb E[X_s^x],$$
for some constant $C>0$. Let now $\Lambda_s^+ = \{u\in \Lambda_s : d(u,\partial H) \ge s\}$, and 
$$X_{s,+}^x = \sum_{u\in \Lambda_s^+} \1\big(x \stackrel{\mathcal H \cap B(0,2s)^c}{\longleftrightarrow} u\big), $$ 
be the corresponding number of pioneers (still seen from $x$). 
By~\cite[Theorem~1]{CH20}, we get that for some constant $c_1>0$, for any $u\in \Lambda_s^+$, 
\begin{equation}\label{tile1} 
\mathbb P(u\stackrel{\mathcal H}{\longleftrightarrow} 0)\ge \frac{c_1}{s^{d-1}}. 
\end{equation}
We now show, using similar arguments as in Section~\ref{sec:pcapball}, that for some constant $c_2>0$, 
\begin{equation}\label{tile2}
 \mathbb P(x\stackrel{\mathcal H}{\longleftrightarrow} 0) \ge \frac{c_2}{s^{d-1}}\cdot \mathbb E[X_{s,+}^x]. 
 \end{equation}
First we show that for some constant $c>0$ and 
 for any set $A\subset \Lambda_s^+$,  
 \begin{equation}\label{tile2bis}
 \mathbb P\big( 0\stackrel{ \mathcal H}{\longleftrightarrow} A\big)\ge \frac{c}{s^{d-1}} \cdot \textrm{Cap}_{d-4}(A).
\end{equation}
For this we follow the proof of Lemma~\ref{lem:hitdiameterdistance}. Let $\nu$ be a probability measure supported on $A$, and $Z = \sum_{a\in A} \nu(a)\cdot \mathbf 1(0\stackrel{ \mathcal H}{\longleftrightarrow}  a)$. One has 
$$ \mathbb P\big( 0\stackrel{ \mathcal H}{\longleftrightarrow} A\big)\ge\mathbb P(Z>0) \ge \frac{\mathbb E[Z]^2}{\mathbb E[Z^2]}.$$  
For the first moment of $Z$, we use~\eqref{tile1}, which shows that $\mathbb E[Z]\gtrsim 1/s^{d-1}$. Concerning the second moment, following the computation in the proof of Lemma~\ref{lem:hitdiameterdistance}, we can see using again~\eqref{tile1} and~\eqref{twopoint} that 
\begin{align*}
\mathbb E[Z^2] & \le \sum_{a,b \in A} \nu(a)\nu(b) \sum_{z \in \mathcal H} \mathbb P(0\stackrel{ \mathcal H}{\longleftrightarrow} z) \mathbb P(z\stackrel{ \mathcal H}{\longleftrightarrow} a)\mathbb P(z\stackrel{ \mathcal H}{\longleftrightarrow} b) \\
& \lesssim  \sum_{a,b \in A} \nu(a)\nu(b) \Big(\sum_{\|z\|\le s/2} s^{2(2-d)} \|z\|^{1-d} + s^{1-d} G(a,b)\Big) \lesssim s^{5-2d} +  \sum_{a,b \in A} \nu(a)\nu(b) G(a,b). 
\end{align*} 
Substituting this above gives 
$$\mathbb P\big( 0\stackrel{ \mathcal H}{\longleftrightarrow} A\big) \ge \frac 1{s^3 + s^{d-1} \sum_{a,b \in A} \nu(a)\nu(b) G(a,b)},$$ 
and as in the proof of Lemma~\ref{lem:hitdiameterdistance}, we deduce~\eqref{tile2bis}. To obtain~\eqref{tile2} from here, we use that by our definition of regularity, (i) the number of $K$-line good points is of the same order as its $(d-4)$-capacity, and (ii) the probability that $0$ connects to them in $\mathcal H$ while avoiding the cluster of $x$ on the outside of the box $B(0,2s)\cap \mathcal H$, is of the same order as the probability that $0$ connects to it in~$\mathcal H$. Finally as observed in Remark~\ref{rem:expreg}, one can see that in expectation the number of $K$-line good points is of the same order as $\mathbb E[X_{s,+}^x]$, which altogether concludes the proof of~\eqref{tile2}.

Then using again the asymptotic from~\cite[Theorem~1]{CH20}, one can see that the left-hand side in~\eqref{tile2} is of order $1/\|x\|^{d-1}$, and thus~\eqref{tile2} gives that for some constant $c_3>0$, 
\begin{equation}\label{eq:important}
\mathbb E[X_{s,+}^x]\le c_3 \cdot \frac{s^{d-1}}{D_x^{d-1}}. 
\end{equation}
Now using translation invariance and a basic coupling, one can see that with $e_1$ the first vector of the canonical basis, and $y = x + se_1$,
\begin{equation*}
\mathbb E[X_s^x -X_{s,+}^x] \le \mathbb E[X_{s,+}^y] \le  c_3 \cdot \frac{s^{d-1}}{D_y^{d-1}}\le   c_4 \cdot \frac{s^{d-1}}{D_x^{d-1}},
\end{equation*}
where $c_4>0$ is another constant. Altogether, this shows that for some constant $C'>0$, 
\begin{equation}\label{Xs+} 
\mathbb P(x\stackrel{\mathcal H}{\longleftrightarrow} Q_s) \le C'\frac{s^{d-3}}{D_x^{d-1}},
\end{equation}
which is of the correct order, at least if $r_x\ge D_x/2$,. Assume next that $D_x \le 4s$. Then~\eqref{onearm} gives for some (possibly different) constant $C>0$, 
\begin{equation}\label{Xs-}
\mathbb P(x\stackrel{\mathcal H}{\longleftrightarrow} Q_s) \le \frac{C}{D_x^2},
\end{equation}
which again is of the correct order if $r_x \ge D_x/2$. It remains to consider the case when $r_x \le D_x/2$. Let $X^x$ be the number of vertices on $\partial B(x,D_x/2)\cap \mathcal H$ which are connected to $x$ 
within $B(x,D_x/2)\cap \mathcal H$. By~\cite[Proposition 2.2]{Romain}, one has for some constant $C_1>0$, 
\begin{equation}\label{rxsmall1}
\mathbb E[X^x] \le C_1 \frac{r_x}{D_x}. 
\end{equation}
Moreover,~\eqref{Xs+} and~\eqref{Xs-} show that for any $u\in \partial B(x,D_x/2)\cap \mathcal H$, 
\begin{equation}\label{rxsmall2}
\mathbb P(u\stackrel{\mathcal H}{\longleftrightarrow} Q_s) \le C_2  \cdot \frac{(s\wedge D_x)^{d-3}}{D_x^{d-1}},
\end{equation}
for some $C_2>0$, and BK inequality shows that 
\begin{equation}\label{rxsmall3}
\mathbb P(x\stackrel{\mathcal H}{\longleftrightarrow} Q_s)\le \mathbb E[X^x] \cdot \sup_{u\in \partial B(x,D_x/2)\cap \mathcal H} \mathbb P(u\stackrel{\mathcal H}{\longleftrightarrow} Q_s). 
\end{equation}
Combining~\eqref{rxsmall1}, \eqref{rxsmall2} and~\eqref{rxsmall3} concludes the proof when $r_x \le D_x/2$. This finishes the proof of the theorem. 
\end{proof}

	
	\section{Connecting two distant sets}\label{sec:connectingdistant}

	In this section we give the proofs of Theorems~\ref{thm.twosets} and~\ref{thm.onearm}. We also prove an additional result, Lemma~\ref{lem:2sets} below, providing a uniform lower bound for the probability to connect two sets in terms of the product of their $(d-4)$-capacities.

	\begin{proof}[\bf Proof of Theorems~\ref{thm.twosets} and~\ref{thm.onearm}]

	Both proofs proceed in a similar way. Concerning Theorem~\ref{thm.twosets}, on the event when $A$ and $z+B$ are connected, consider the last pivotal oriented edge $\vec e=(x,y)$ defined by the following conditions: first $x$ is connected to $A$ by two disjoint open paths, secondly $\{x,y\}$ is open but closing it disconnects $A$ from $B$, and finally $y$ is connected to $B$. Denote by $\mathcal P_{z,A,B}(\vec e)$ the corresponding event. It may happen that such an edge does not exist, but in this case one can find two disjoint open paths connecting $A$ to $z+B$, which has probability of order at most $(|A|\cdot|B| \cdot \tau(z))^2$, using BK inequality~\eqref{BKineq}, and is thus an  event with negligible probability. Now for each fixed oriented edge $\vec e=(x,y)$, one may localize the event $\mathcal P_{z,A,B}(\vec e)$, similarly as in the proof of Theorem~\ref{thm.defpcap}). Then using the result of~\cite{CCHS25} stating that the probability of any local event, conditionally on a point $y$ being connected to any arbitrary set going to infinity, converges in law to the same event under the law of the IIC rooted at $y$, we deduce that 
	$$\lim_{\|z\|\to \infty} \frac{\mathbb P(\mathcal P_{z,A,B}(\vec e))}{\mathbb P(y\longleftrightarrow z+B)} = \mathbb P(\mathcal P_{\infty,A}(\vec e)), $$ 
	where $\mathcal P_{\infty,A}(\vec e)$ is the event that in $\mathcal C_\infty(y)$ one can find two disjoint open paths from $x$ to $A$, and when one closes the edge $e$, $y$ gets disconnected from $A$ in $\mathcal C_\infty(y)$.  
	On the other hand we know from Theorem~\ref{thm.defpcap} that for any fixed $y$, 
	$$\lim_{\|z\|\to \infty} \frac{\mathbb P(y\longleftrightarrow z+B)}{\tau(z)} = \textrm{pCap}(B), $$ 
	and from~\eqref{expr.pivotal.pcap} that
	$$\sum_{\vec e} \mathbb P(\mathcal P_{\infty, A}(\vec e)) = \textrm{pCap}(A).  $$
	Then Theorem~\ref{thm.twosets} follows.
	The proof of Theorem~\ref{thm.onearm} is entirely similar. More precisely, one may define $\mathcal P_{r,A}(\vec e)$, as the event that $x$ is connected to $A$ using two disjoint open paths, $\{x,y\}$ is open, closing $e$ disconnects $A$ from $\partial B(0,r)$, and $y$ is connected to $\partial B(0,r)$. Then exactly the same argument as above proves the desired result, just using additionally the one arm estimate~\eqref{onearm}. 
		\end{proof}

	\begin{lemma}\label{lem:2sets}
		For every constant $c_1>0$, there exists a constant $c>0$ so that if $A, B\subseteq \Z^d$ are two finite sets containing $0$, then for all $z$ with $d(z,A), d(z,B)\geq c_1 \max_{a\in A,b\in B}\|a-b\|$ we have 
		\[
		\pr{A\longleftrightarrow B+z} \geq c \cdot \tau(z) \cdot \cpc{d-4}{A}\cpc{d-4}{B}.
		\]
	\end{lemma}

	\begin{remark}
		\rm{
		Suppose $A$ and $B$ are two sets satisfying the assumptions of Claim~\ref{cl:lowerboundoncapacity}. Then their $(d-4)$-capacities will be of order their cardinality, and hence for such sets the statement of the lemma above is sharp up to constant factors. In~\cite[first item of Lemma~3.2]{CCHS25} a similar lower bound is obtained for sets that satisfy a regularity property as defined in~\cite{KN11} (see also Section~\ref{sec.regpoints} for our simplified definition). Imposing the extra condition on the number of points in balls on the boundary of the cube, the proof of the lower bound on the connection probability between two distant sets becomes considerably shorter. 
		}
	\end{remark}

	\begin{proof}[\bf Proof of Lemma~\ref{lem:2sets}]
		
		The proof of this lemma is very similar to the proof of Lemma~\ref{lem:hitdiameterdistance} and we just explain the appropriate changes. 
		Let $\mu$ and $\nu$ be two probability measures supported on $A$ and $B$ respectively. We define
		\[
		Z=\sum_{x\in A}\sum_{y\in B} \mu(x) \nu(y) \1(x\longleftrightarrow y+z).
		\]
		Then we clearly have 
		\begin{align}\label{eq:firstineqpayley}
			\pr{A\longleftrightarrow B+z} \geq \pr{Z>0} \geq \frac{(\E{Z})^2}{\E{Z^2}},
		\end{align}
		where the second inequality follows by the Payley-Zygmund inequality. 
		Standard computations as in Lemma~\ref{lem:hitdiameterdistance} and also in Lemma~\ref{lem:secondmomentofy1} give that 
		\[
		\E{Z} \asymp \tau(z)
		\]
		and for the second moment 
	\begin{align*}
		\E{Z^2} &\lesssim \tau(z) \sum_{x\in A,y\in B}\mu^2(x)\nu^2(y) +  \tau(z)\sum_{x,x'\in A} \mu(x) \mu(x') G(x,x')\\
		&+\tau(z)\sum_{x,x'\in A} \sum_{y,y'\in B} \mu(x) \mu(x') \nu(y)\nu(y') G(x,x') G(y,y')\\
		&+ (\tau(z))^2 \sum_{x,x'\in A} \mu(x) \mu(x')G(x,x') + (\tau(z))^2 \sum_{y,y'\in B} \nu(y) \nu(y')G(y,y').
		\end{align*}
		Writing $\cE(p,p)=\sum_{x,x'} p(x)p(x')G(x,x')$ for $p$ a probability measure and using that
		\[
		\cE(\mu,\mu)\gtrsim (\max_{a\in A}\|a\|)^{4-d} \quad \text{ and } \quad \cE(\nu,\nu)\gtrsim (\max_{b\in B}\|b\|)^{4-d},
		\]
		we obtain that the dominant term in the upper bound for $\E{Z^2}$ is $\tau(z) \cdot \cE(\mu,\mu)\cdot \cE(\nu,\nu)$, and hence plugging this into~\eqref{eq:firstineqpayley} and optimising over the choice of the probability measures $\mu$ and $\nu$ finishes the proof. 		
		\end{proof}

	\textbf{Acknowledgments:} We thank Romain Panis for insightful comments, in particular regarding the equilibrium measure described in Remark~\ref{rem.eqmes.ordering},  and sharing with us his proof of the improved bound on the two-point function restricted to the half space (first item of Lemma~\ref{lem.LD}, which is now part of~\cite{Romain}). We also thank David Croydon for helpful comments.

\end{document}